\documentclass[11pt,letter,reqno]{amsart}
\usepackage[top=2.5cm, bottom=2.5cm, inner=2.5cm, outer=2.5cm, marginparwidth=2cm]{geometry}

\usepackage{amssymb, latexsym, amsmath, amsxtra, mathrsfs, bm}

\usepackage[dvips]{graphics}
\usepackage{xypic,xcolor}
\usepackage{subcaption}

\usepackage{verbatim}
\usepackage[abs]{overpic}
\usepackage{hyperref}
\usepackage{mathtools}

\usepackage{enumitem}

\DeclareMathAlphabet{\mathbbold}{U}{bbold}{m}{n}

\allowdisplaybreaks[3]

\theoremstyle{plain}
        \newtheorem{theorem}{Theorem}[section]
        \newtheorem*{theorem*}{Theorem}
        \newtheorem*{conj*}{Conjecture}
        \newtheorem{lemma}[theorem]{Lemma}
        \newtheorem{prop}[theorem]{Proposition}
        \newtheorem*{prop*}{Proposition}
        \newtheorem{cor}[theorem]{Corollary}
        \newtheorem*{cor*}{Corollary}

        \newtheorem{thmx}{Theorem}

\theoremstyle{definition}
        \newtheorem{definition}[theorem]{Definition}
        \newtheorem*{definition*}{Definition}
        \newtheorem{rem}[theorem]{Remark}

\theoremstyle{remark}
        \newtheorem*{remark}{Remark}

\numberwithin{equation}{section}
\numberwithin{theorem}{section}
\numberwithin{table}{section}
\numberwithin{figure}{section}

\makeatletter
\def\th@plain{%
	\thm@notefont{}
	\itshape 
}
\def\th@definition{%
	\thm@notefont{}
	\normalfont 
}
\makeatother

\providecommand{\defn}[1]{\emph{#1}}

\newcommand{\inter}  {\operatorname{int}}

\newcommand{\id} {\operatorname{id}}

\newcommand{\card} {\operatorname{card}}
\newcommand{\R}{\mathbb{R}}
    
\newcommand{\C}{\mathbb{C}}

\newcommand{\N}{\mathbb{N}}      
\newcommand{\Z}{\mathbb{Z}}      

\renewcommand{\le}{\leqslant}
\renewcommand{\leq}{\leqslant}
\renewcommand{\ge}{\geqslant}
\renewcommand{\geq}{\geqslant}

\renewcommand{\:}{\colon}

\providecommand{\abs}[1]{\lvert#1\rvert}
\providecommand{\Absbig}[1]{\bigl\lvert#1\bigr\rvert}

\providecommand{\norm}[1]{\|#1\|}

\renewcommand{\=}{\coloneqq}

\newcommand{\cB}{\mathcal{B}}
\newcommand{\cC}{\mathcal{C}}
\renewcommand{\cD}{\mathcal{D}}

\newcommand{\cF}{\mathcal{F}}

\renewcommand{\cH}{\mathcal{H}}
\newcommand{\cI}{\mathcal{I}}

\newcommand{\cK}{\mathcal{K}}

\newcommand{\cM}{\mathcal{M}}
\newcommand{\cN}{\mathcal{N}}

\newcommand{\cP}{\mathcal{P}}

\newcommand{\fX}{\mathfrak{X}}
\newcommand{\fY}{\mathfrak{Y}}

\newcommand{\tgamma}{\widetilde{\gamma}}

\newcommand{\ttau}{\widetilde{\tau}}

\newcommand{\Sp}{\mathbb{S}}
\newcommand{\cellint}{\operatorname{int_\circ}}
\newcommand{\cellbound}{\partial_{\circ}}

\newcommand{\cls}[1]{\overline{#1}}

\newcommand{\loc}{{\operatorname{loc}}}
\newcommand{\celPair}{cellular pair}
\newcommand{\mfd}{\cM}
\newcommand{\dg}{d_g}
\newcommand{\Ko}{\cK^{\circ}}
\newcommand{\Rec}{\mathrm{Rec}}
\newcommand{\Cube}{\mathrm{Cube}}
\newcommand{\stdCube}{\mathbb{I}}
\newcommand{\Isom}{\mathrm{Isom}}
\newcommand{\Sym}{\mathrm{Sym}}
\newcommand{\diag}{\mathrm{diag}}
\newcommand{\Cone}{\mathrm{Cone}}
\newcommand{\ind}{i}
\newcommand{\stdInt}[1]{\inter #1}
\newcommand{\stdCellint}[1]{\cellint #1}
\newcommand{\brCellint}[1]{\cellint(#1)}
\newcommand{\fundSet}{fundamental subset}
\newcommand{\sk}[2]{\bigcup_{i\leq #2}#1^{[i]}}
\newcommand{\clcnR}{connected closed Riemannian}

\newcommand{\orclcnR}{oriented, connected, and closed Riemannian}

\newif\ifneworder
\newordertrue

\begin{document}

\title{A class of Latt\`es maps with cellular structures}

\author{Zhiqiang Li and Hanyun Zheng}

\address{Zhiqiang~Li, School of Mathematical Sciences \& Beijing International Center for Mathematical Research, Peking University, Beijing 100871, China.}
\email{zli@math.pku.edu.cn}

\address{Hanyun~Zheng, School of Mathematical Sciences, Peking University, Beijing 100871, China.}
\email{1900013001@pku.edu.cn}



\begin{abstract} 
We show that a class of quasiregular Latt\`es maps, called orthotopic Latt\`es maps, are cellular Markov maps. This provides examples of expanding Thurston-type maps that are also uniformly quasiregular, and whose visual metrics are quasisymmetrically equivalent to the Riemannian distance.
\end{abstract}

\maketitle

{
    \hypersetup{linkcolor=black}
    \setcounter{tocdepth}{2}
    \tableofcontents
}


\section{Introduction}\label{sct: Intro}
As a higher-dimensional analog of the dynamics of rational maps on the Riemann sphere $\widehat\C$, the theory of uniformly quasiregular (abbreviated as UQR) maps has been developed in the past decades. Heuristically, a
quasiregular map is a map $\mfd^n\to\mfd^n$ with bounded distortion on an {\orclcnR} $n$-manifold $\mfd^n$, and a uniformly quasiregular map is a quasiregular map whose iterates have a uniform bound of distortion. We recall these notions in Section~\ref{sct: prelim} and refer the reader to \cite{IM01,Ka21,Ri93,Vu88} for the general quasiregular theory.

An important subclass of uniformly quasiregular maps are the quasiregular Latt\`es maps, which include the well-known Latt\`es maps on $\widehat\C$. A \defn{quasiregular Latt\`es map} is a uniformly quasiregular map $f\:\mfd^n\to\mfd^n$ that is semi-conjugate to a conformal affine map $A\:\R^n\to\R^n$ via a quasiregular map $h\:\R^n\to\mfd^n$ strongly automorphic with respect to some discrete subgroup $\Gamma<\Isom(\R^n)$, as in the following commutative diagram.
\begin{equation*}
	\xymatrix{
		\R^n \ar[r]^A \ar[d]_h & \R^n \ar[d]^h \\
		\mfd^n \ar[r]^f & \mfd^n
	}
\end{equation*} 
We give a precise definition in Section~\ref{sct: prelim}, and refer to \cite{IM01,May97,Ka22} for more information.

In dimension two, Latt\`es maps on $\widehat\C$ belong to the class of Thurston maps, i.e., branched covering maps with finite postcritical points and degree at least $2$; cf.~\cite{BM17} for an exposition of the theory of Thurston maps. In the higher-dimensional case, it is natural to investigate whether (quasiregular) Latt\`es maps have properties analogous to Thurston maps.

A class of higher-dimensional analogs of Thurston maps called \emph{Thurston-type maps} were introduced in \cite{LPZ25}. The definition of Thurston-type maps is based on topological and combinatorial conditions (see Definition~\ref{d: CPCF}), which naturally generalize the postcritically-finite condition of Thurston maps, and lead to a higher-dimensional theory having potential for fruitful results. In particular, \cite[Theorem~A]{LPZ25} states the following rigidity phenomenon reminiscent of the No Invariant Line Fields conjecture (see e.g.~\cite{McS98}): \emph{In dimension $n\geq 3$, if an expanding Thurston-type map is uniformly quasiregular, then it is a Latt\`es map.}

In view of the discussion above,
one may ask, in higher dimensions, which (quasiregular) Latt\`es maps are Thurston-type maps. In this article, we show that this holds for a class of Latt\`es maps, which we call orthotopic Latt\`es maps.
\begin{thmx}\label{tx: orthotopic Lattes}
	An orthotopic Latt\`es map $\mfd^n \to \mfd^n$, $n\ge3$, on a {\clcnR} $n$-manifold $\mfd^n$ is a cellular Markov map. Moreover, if $f$ has topological degree at least two, then $f$ is a Thurston-type map.
\end{thmx}
We give a formal definition of orthotopic Latt\`es maps in Section~\ref{sct: rectangle-Lattes}. Heuristically, an orthotopic Latt\`es map is a Latt\`es map where the group $\Gamma$ has a fundamental domain in the shape of an orthotope. Although this class of Latt\`es maps has not been 
given a name in the UQR literature, 
many known examples of UQR maps on closed Riemannian manifolds are orthotopic; see e.g. \cite{AKP10} for a discussion on such examples.

We recall the definitions of cellular Markov maps and Thurston-type maps in Section~\ref{sct: prelim}, as well as briefly discuss their relation; see also \cite[Chapter~5]{BM17} and~\cite[Section~1]{LPZ25}. Since the theory of Thurston-type maps is beyond the scope of this article, we will not include a detailed discussion.

If an orthotopic Latt\`es map $f\:\mfd^n\to\mfd^n$, $n\ge 3$, is \emph{chaotic} (see Definition~\ref{d: Lattes triple}), then it is an expanding Thurston-type map, which naturally induces a class of metrics called \emph{visual metrics of $f$}. By \cite[Theorem~B]{LPZ25}, these visual metrics are quasisymmetrically equivalent to the Riemannian distance on $\mfd^n$. See~\cite[Sections~1--3 and Appendix~A.2]{LPZ25} (cf. \cite{BM17,HP09}) for a discussion on these notions.
\begin{cor*}
	Let $f\:(\mfd^n,\dg)\to(\mfd^n,\dg)$, $n\ge 3$, be a chaotic orthotopic Latt\`es map on a {\clcnR} $n$-manifold $(\mfd^n,\dg)$. Then each visual metric of $f$ is quasisymmetrically equivalent to $\dg$.
\end{cor*}

Finally, we note that it is not known, to our knowledge, whether all Latt\`es maps on $\mfd^n$ with dimension $n \geq 3$ are Thurston-type maps, or even have cellular branch sets. For general UQR maps, this is not the case. Indeed, by a result of Martin and Peltonen \cite{MarPe10}, the branch set of a quasiregular map on $\Sp^n$ is realizable as the branch set of a uniformly quasiregular map on $\Sp^n$,
and by the construction of Heinonen and Rickman \cite{HR98} there exist quasiregular maps $\Sp^3 \to \Sp^3$ without cellular branch sets.

\medskip
\noindent{\bf Acknowledgements.} 
The authors would like to express their sincere gratitude to  
Pekka~Pankka for helpful discussions. 
Z.~Li and H.~Zheng were partially supported by Beijing Natural Science Foundation (JQ25001, 1214021) and National Natural Science Foundation of China (12471083, 12101017, 12090010, and 12090015).

\section{Preliminaries}\label{sct: prelim}
\subsection{Quasiregular maps and Latt\`es maps}\label{subsct: prelim on BC}
Recall that a continuous map $f\colon \mfd^n \to \cN^n$ between oriented Riemannian $n$-manifolds is \defn{quasiregular} if $f$ is in the Sobolev space $W^{1,n}_{\loc}(\mfd^n,\cN^n)$ and there exists $K \geq 1$ for which the distortion inequality
\begin{equation}	\label{eq:K}
	\norm{Df}^n  \leq  K J_f \quad \text{a.e.\ }\mfd^n
\end{equation}
holds, where $\norm{Df}$ and $J_f$ are the operator norm and the Jacobian, respectively, of the differential $Df$ of $f$. In this case, we say that $f$ is \defn{$K$-quasiregular}. In particular, a holomorphic map of one complex variable is a $1$-quasiregular map. In this terminology, a map is \defn{quasiconformal} if it is a quasiregular homeomorphism. 
We refer to the monographs of Rickman \cite{Ri93}, Reshetnyak \cite{Re89}, and Iwaniec \& Martin \cite{IM01} for detailed expositions on quasiregular theory.

Following Iwaniec \& Martin \cite{IM96} (see also \cite{IM01}), a quasiregular map $f\colon \mfd^n\to\mfd^n$ of an oriented Riemannian $n$-manifold is \defn{uniformly quasiregular} (abbreviated as \defn{UQR}) if there exists $K \geq  1$ for which each iterate $f^k$, $k\geq 1$, of $f$ is $K$-quasiregular.
\footnote{Note that the uniform quasiregularity of a map is independent of the choice of the Riemannian metric of $\mfd^n$, since, for Riemannian metrics $g$ and $g'$ on $\mfd^n$, the identity map $\id \colon (\mfd^n,g) \to (\mfd^n,g')$ of a closed Riemannian $n$-manifold is quasiconformal.}
See a survey of Martin \cite{Mar14} for a detailed discussion.

An important subclass of uniformly quasiregular maps are the quasiregular Latt\`es maps originating from the work of Mayer~\cite{May97,May98}.
For the convenience of the reader, we recall the notion of Latt\`es maps (cf.~\cite{IM01,Ka22,May97} for more information).
\begin{definition}\label{d: Lattes triple}
	Let $\mfd^n$, $n\ge 3$, be an {\orclcnR} $n$-manifold. A triple $(\Gamma,h,A)$ is a \defn{Latt\`es triple} on $\mfd^n$ if the following conditions are satisfied: 
	\begin{enumerate}[label=(\roman*),font=\rm]
		\smallskip
		\item $\Gamma$ is a discrete subgroup of the isometry group $\Isom(\R^n)$.
		\smallskip
		\item $h\:\R^n\to\mfd^n$ is a quasiregular map which is strongly automorphic with respect to $\Gamma$, i.e., for all $x,\,y\in\R^n$, $h(x)=h(y)$ if and only if $y=\gamma(x)$ for some $\gamma\in\Gamma$.
		\smallskip
		\item $A$ is a conformal affine map (i.e.,~$A=\lambda U+v$ for some $\lambda>0$, $U\in \operatorname{O}(n)$, and $v\in\R^n$) with $A\Gamma A^{-1}\subseteq\Gamma$.
	\end{enumerate}
	A \defn{Latt\`es map} is a uniformly quasiregular map $f\:\mfd^n\to\mfd^n$ for which there exists a Latt\`es triple $(\Gamma,h,A)$ with
	$f\circ h=h\circ A$.
	In this case, we call $f$ a \defn{Latt\`es map with respect to $(\Gamma,h,A)$}.
	We call a Latt\`es map with respect to a Latt\`es triple $(\Gamma,h,A)$ \emph{chaotic} if $\Gamma$ is cocompact and $A$ is expanding, i.e., $A=\lambda U+v$ for some $\lambda>1$, $U\in \operatorname{O}(n)$, and $v\in\R^n$.
\end{definition}

\subsection{Cellular Markov maps}\label{subsct: cell decomp}
Here we recall the notion of cell decompositions. Our definitions of cells and cell decompositions follow \cite[Chapter~5]{BM17}. Throughout the remaining part of this section, $\fX$ will always be a locally-compact Hausdorff space. 

Recall that, for each $n\in\N$, a subset $c$ of $\fX$ homeomorphic to $[0,1]^n$ is called an \defn{$n$-dimensional cell}, and $\dim(c)\=n$ is called the dimension of $c$. We denote by $\cellbound c$ the set of points corresponding to $[0,1]^n\smallsetminus (0,1)^n$ under a homeomorphism between $c$ and $[0,1]^n$. We call $\cellbound c$ the \defn{cell-boundary} and $\stdCellint{c}\=c\smallsetminus\cellbound c$ the \defn{cell-interior} of $c$. The sets $\cellbound c$ and $\stdCellint{c}$ are independent of the choice of the homeomorphism. Note that the cell-boundary and cell-interior generally do not agree with the boundary $\partial c$ and interior $\stdInt{c}$ of $c$ regarded as a subset of the topological space $\fX$. A $0$-dimensional cell is a subset consisting of a single point in $\fX$. For a $0$-dimensional cell $c$, we set $\cellbound c\=\emptyset$ and $\stdCellint{c}\=c$.

We recall some basic properties of cells for further discussion, omitting the standard proofs. 

\begin{lemma}\label{l: properties of cells}
	The following statements are true:
	\begin{enumerate}[label=(\roman*),font=\rm]
    	\smallskip
		\item Let $\fX$ be a locally-compact Hausdorff space. Then each cell $c\subseteq\fX$ is closed. Moreover, $c=\cls{\stdCellint{c}}$.
		\smallskip
		\item Let $\fX$ be a topological $n$-manifold. Then for each $n$-dimensional cell $X$ in $\fX$, $\stdCellint{X}$ and $\cellbound X$ agree with the interior and boundary of $X$ regarded as a subset of the topological manifold $\fX$, respectively.
	\end{enumerate}
\end{lemma}

\begin{definition}[Cell decompositions]\label{d: cell decomposition}
A collection $\cD$ of cells in a locally-compact Hausdorff space $\fX$ is a \defn{cell decomposition of $\fX$} 
if the following conditions are satisfied:
    \begin{enumerate}[label=(\roman*),font=\rm]
    	\smallskip
        \item The union of all cells in $\cD$ is equal to $\fX$.
        \smallskip
        \item $\brCellint{\sigma}\cap\brCellint{\tau}=\emptyset$ for all distinct $\sigma,\,\tau\in\cD$.
        \smallskip
        \item For each $\tau\in\cD$, the cell-boundary $\cellbound\tau$ is a union of cells in $\cD$.
        \smallskip
        \item Every point in $\fX$ has a neighborhood that meets only finitely many cells in $\cD$.
    \end{enumerate}
\end{definition}

For a collection $\cC$ of cells in an ambient space $\fX$, we denote by $\abs{\cC}\=\bigcup\cC$ the \defn{space} of $\cC$. If $\cC$ is a cell decomposition of $\abs{\cC}$, then we call $\cC$ a \defn{cell complex}. A cell decomposition is a cell complex.
A subset $\cC'\subseteq\cC$ is a \defn{subcomplex} of a cell complex $\cC$ if $\cC'$ is a cell complex.

Let $\cD$ be a cell decomposition of $\fX$.
For each $S \subseteq \fX$, we denote 
\begin{equation}\label{e: restriction of cell decomp}
\cD|_{S}\=\{c \in\cD : c \subseteq S \},
\end{equation}
and call it the \defn{restriction of $\cD$ on $S$}. 
For each $k\in\N_0$, we denote
 \begin{equation*}
 \cD^{[k]}\=\{c\in\cD:\dim(c)=k\},
 \end{equation*} 
 and call $\Absbig{\sk{\cD}{k}}$ the \defn{$k$-skeleton of $\cD$}.

We record some elementary properties of cell decompositions.

\begin{lemma}[{\cite[Lemmas~5.2 and~5.3]{BM17}}]\label{l: properties of cell decompositions}
	Let $\cD$ be a cell decomposition of $\fX$. Then the following statements are true:
	\begin{enumerate}[label=(\roman*),font=\rm]
    	\smallskip
		\item For each $k\in\N_0$, the $k$-skeleton of $\cD$ is equal to $\bigcup\{\stdCellint{c}:c\in\cD,\,\dim(c) \leq  k\}$.
		\smallskip
		\item $\fX=\bigcup\{\stdCellint{c}:c\in\cD\}$.
		\smallskip
		\item For each $\tau\in\cD$, we have $\tau=\bigcup\{\stdCellint{c}:c\in\cD,\,c\subseteq\tau\}$.
		\smallskip
		\item If $\sigma$ and $\tau$ are two distinct cells in $\cD$ with $\sigma\cap\tau\neq\emptyset$, then one of the following three statements is true: $\sigma\subseteq\cellbound\tau$, $\tau\subseteq\cellbound\sigma$, or $\sigma\cap\tau=\cellbound\sigma\cap\cellbound\tau$ and the intersection consists of cells in $\cD$ of dimension strictly less than $\min\{\dim(\sigma),\dim(\tau)\}$.
		\smallskip
		\item If $\sigma,\,\tau_1,\,\dots,\,\tau_k$ are cells in $\cD$ and $\brCellint{\sigma}\cap\bigcup_{j=1}^k\tau_j\neq\emptyset$, then $\sigma\subseteq\tau_i$ for some $i\in\{1,\,\dots,\,k\}$.
	\end{enumerate}
\end{lemma}

\begin{lemma}[{\cite[Lemma~2.10]{LPZ25}}]\label{l: cell decomp: restrict}
     Let $\cD$ be a cell decomposition of a locally-compact Hausdorff space $\fX$ and $S\subseteq\fX$. Then the following statements are true:
     \begin{enumerate}[label=(\roman*),font=\rm]
        \smallskip
         \item  $\cD|_S$ is a subcomplex of $\cD$.
         \smallskip
         \item If $S$ is a union of cells in $\cD$, then $\cD|_S$ is a cell decomposition of $S$. In particular, for each $c\in\cD$, $\cD|_c$ and $\cD|_{\cellbound c}$ are cell decompositions of $c$ and $\cellbound c$, respectively.
         \smallskip
         \item If $S=\abs{\cC}$ for a subcomplex $\cC$ of $\cD$, then $\cD|_S=\cC$. 
     \end{enumerate}
 \end{lemma}

We also recall the notion of \emph{refinements of cell decompositions}, and record several properties.
\begin{definition}[Refinements]\label{d: refinement}
A cell decomposition $\cD_1$ of a locally-compact Hausdorff space $\fX$ is a \defn{refinement} of a cell decomposition $\cD_0$ of $\fX$ if the following conditions are satisfied:
\begin{enumerate}[label=(\roman*),font=\rm]
	\smallskip
    \item For each $\sigma\in\cD_1$, there exists $\tau\in\cD_0$ satisfying $\sigma\subseteq\tau$.
    \smallskip
    \item Each cell $\tau\in\cD_0$ is the union of cells $\sigma\in\cD_1$ satisfying $\sigma\subseteq\tau$.
\end{enumerate}
In this case, we also say that $\cD_1$ \defn{refines $\cD_0$}.
\end{definition}

\begin{lemma}[{\cite[Lemma~5.7]{BM17}}]\label{l: cell decomp: refinement inte() contained in inte()}
	Let $\cD$ be a cell decomposition of a locally-compact Hausdorff space $\fX$ and let $\cD'$ be a refinement of $\cD$. Then for each $\sigma\in\cD'$, there exists a minimal cell $\tau\in\cD$ with $\sigma\subseteq\tau$, i.e., if $\sigma\subseteq\ttau$ for some $\ttau\in\cD$, then $\tau\subseteq\ttau$. Moreover, $\tau$ is the unique cell in $\cD$ with $\stdCellint{\sigma}\subseteq\stdCellint{\tau}$.
\end{lemma}

\begin{lemma}[{\cite[Lemma~2.17]{LPZ25}}]\label{l: construct refinement}
	Let $\cD$ be a cell decomposition of a locally-compact Hausdorff space $\fX$. For each $c\in\cD$, let $\cD'(c)$ be a cell decomposition of $c$ that refines $\cD|_c$. Then $\cD'\=\bigcup_{c\in\cD}\cD'(c)$ is a cell decomposition that refines $\cD$ if and only if for all $c,\,\sigma\in\cD$ with $c\subseteq\sigma$, we have $\cD'(\sigma)|_c=\cD'(c)$.
\end{lemma}

Now we recall cellular maps and related notions; see~\cite[Chapter~5]{BM17} or \cite[Section~2]{LPZ25} for more information.
In what follows, $\fX$ and $\fX'$ stand for locally-compact Hausdorff spaces.

\begin{definition}[Cellular maps]\label{d: cellular map}
	A continuous map $f\:\fX'\to\fX$ is a \defn{cellular map} if there exist cell decompositions $\cD'$ and $\cD$ of $\fX'$ and $\fX$, respectively, such that for each $c\in\cD'$, $f(c)$ is a cell in $\cD$ and the restriction $f|_{c}\:c\to f(c)$ is a homeomorphism. In this case, we call $f$ a \defn{$(\cD',\cD)$-cellular map}, $(\cD',\cD)$ a \defn{{\celPair}} of $f$, and $\cD'$ a \defn{pullback} of $\cD$ by $f$.
\end{definition}

Note that the pullback $\cD'$ of cell decomposition $\cD$, if exists, is uniquely determined by $\cD$ (see \cite[Subsection~2.3]{LPZ25} for a discussion), although neither each cell decomposition admits a pullback, nor {\celPair}s in Definition~\ref{d: cellular map} are unique.

Homeomorphisms are cellular maps: under a homeomorphism, each cell decomposition admits a unique pullback. For a homeomorphism  $\phi\:\fX'\to\fX$ and a cell decomposition $\cD$ of $\fX$, denote
\begin{equation}\label{e: homeo pullback}
	\phi^*(\cD)\=\bigl\{\phi^{-1}(c):c\in\cD\bigr\}.
\end{equation}
Then it is easy to check that $\phi^*(\cD)$ is a cell decomposition.
\begin{lemma}\label{l: cellular: cellular homeo}
	If $\phi\:\fX'\to\fX$ is a homeomorphism and $\cD$ is a cell decomposition of $\fX$, then $\phi^*(\cD)\=\bigl\{\phi^{-1}(c):c\in\cD\bigr\}$ is a cell decomposition, and the unique pullback of $\cD$.
	In particular, given a $(\cD',\cD)$-cellular map $f\colon\fX'\to\fX$, for each $c'\in \cD'$, we have $\cD'|_{c'}=(f|_{c'})^*\bigl(\cD|_{f(c')}\bigr)$. 
\end{lemma}

We conclude with the definition of the aforementioned cellular Markov maps and Thurston-type maps.
\begin{definition}[Cellular Markov partitions, cellular Markov maps]\label{d: cellular Markov}
	A \defn{cellular Markov partition} of a continuous map $f\:\fX\to\fX$ is a {\celPair} $(\cD',\cD)$ of $f$ where $\cD'$ is a refinement of $\cD$. We call $f$ a \defn{cellular Markov map} if there exists a cellular Markov partition of $f$.
\end{definition}

In this article, we adopt the formulation in \cite{HR02} and say that a discrete, open, and continuous map $f\colon \fX\to \fY$ between topological spaces $\fX$ and $\fY$ is a \defn{branched cover}. Note that quasiregular maps are orientation-preserving branched covers (see e.g.\ \cite[Chapter~VI]{Ri93} or \cite[Section~11]{Ka21}). For a branched cover $f\:\fX\to\fY$, we denote the \defn{branch set} (or \defn{critical set}) and \defn{postbranch set} (or \defn{postcritical set}), respectively, by
\begin{equation*}
    B_f\=\{x\in\fX:f\mbox{ is not a local homeomorphism at }x\}\quad\mbox{and}\quad P_f\=\bigcup_{m\geq 1}f^m(B_f).
\end{equation*}
We denote the \defn{local multiplicity} of $f$ at $x\in\fX$ by 
\begin{equation*}
\ind(x,f)\=\inf\{N(f,U):U\subseteq\fX\mbox{ is an open neighborhood of }x\},	
\end{equation*}
where $N(f,U)\=\sup\{\card\bigl(f^{-1}(y)\cap U\bigr):y\in f(U)\}$

\begin{definition}[Thurston-type maps]\label{d: CPCF}
    A branched cover $f\:\mfd^n\to\mfd^n$ is \defn{cellularly-postcritically-finite (CPCF)} if the following conditions hold:
    \begin{enumerate}[label=(\roman*),font=\rm]
        \smallskip
        \item There exists a cellular Markov partition $(\cP_0,\cP)$ of $f|_{P_f}\:P_f\to P_f$ and a cell decomposition $\cD_0$ of $\mfd^n$ such that $\cP_0\subseteq\cD_0$.
        \smallskip
        \item There exists a cell decomposition $\cB$ of $B_f$ such that $f|_{B_f}\:B_f\to P_f$ is $(\cB,\cP)$-cellular, and that for each $c\in\cB$, the function $i(f,\cdot)|_{\stdCellint{c}}\:\stdCellint{c}\to\N,\,x\mapsto i(x,f)$ is constant.
    \end{enumerate}
    A \defn{Thurston-type map} is a CPCF branched cover with topological degree at least two.
\end{definition}

\begin{rem}\label{r: Markov => Thurston-type}
    We note that cellular Markov branched covers are CPCF (see~\cite[Proposition~3.21]{LPZ25}), and are of Thurston-type when having topological degrees at least two; cf.~\cite[Section~3]{LPZ25} for a detailed discussion.
\end{rem}

\section{Orthotopic Latt\`es maps}\label{sct: rectangle-Lattes}

The main purpose of this article is to show that orthotopic Latt\`es maps are cellular Markov.

\begin{theorem}\label{t: rectangle=>Markov}
	An orthotopic Latt\`es map $\mfd^n\to\mfd^n$, $n\ge 3$, on a {\clcnR} $n$-manifold is a cellular Markov map.
\end{theorem}

Consequently, since Latt\`es maps are branched covers and by Remark~\ref{r: Markov => Thurston-type}, orthotopic Latt\`es maps with degree at least two are Thurston-type maps. Then Theorem~\ref{tx: orthotopic Lattes} follows.

Now, we proceed to the definition of orthotopic Latt\`es maps and discuss first the underlying notion of orthotopic crystallographic groups.

\subsection{Orthotopic crystallographic groups and Latt\`es triples}
{In this subsection, we fix an arbitrary $n\in\N$ and use the convention that all discussions take place in the ambient space $\R^n$.}

First, following~\cite{SV93}, we recall some basic notions on crystallographic groups and Euclidean geometry of $\R^n$. A discrete subgroup $\Gamma$ of $\Isom(\R^n)$ is called \defn{crystallographic} if $\Gamma$ acts cocompactly on $\R^n$, i.e., the quotient space $\R^n/\Gamma$ is compact. A \defn{fundamental domain} of a discrete subgroup $\Gamma<\Isom(\R^n)$ is a closed domain (i.e.,~the closure of some domain) $D\subseteq\R^n$ satisfying the following properties:
\begin{enumerate}[label=(\roman*),font=\rm]
	\smallskip
	\item $\{\gamma(D):\gamma\in\Gamma\}$ is a locally-finite cover of $\R^n$.
	\smallskip
	\item If $\gamma\in\Gamma$ and $\gamma\neq\id$, then $\stdInt{(D)}\cap\stdInt{(\gamma(D))}=\emptyset$.
\end{enumerate}

Note that for a crystallographic group $\Gamma$ there always exists a \defn{normal fundamental polyhedron}, i.e., a fundamental domain $P$ that is a convex polyhedron, and the intersection of two adjacent polyhedra in $\{\gamma(P):\gamma\in\Gamma\}$ is a common face of them. Examples of polyhedra include simplices, cubes, orthotopes, etc. In the current section, we only consider orthotopes and do not discuss polyhedra in detail; cf.~\cite{SV93} for a more comprehensive discussion.

In what follows, for $d\in\N_0$ with $d\leqslant n$, we identify $\R^d$ with the subspace $\R^d\times\{0\}^{n-d}$ of $\R^n$. We also use the notation $0$ to denote the point $(0,\,\dots,\,0)\in\R^n$ when there is no ambiguity.

Fix $d\in\N$ with $d\le n$. We call a subset of $\R^d$, having the form $\prod_{i=1}^d[-a_i,a_i]$, $a_i\in(0,+\infty)$ for each $1\leq i\leq d$, a \defn{standard $d$-dimensional orthotope}.  
A subset $Q\subseteq\R^n$ is a \defn{$d$-dimensional orthotope}  if $Q$ is isometric to a standard $d$-dimensional orthotope. A \defn{$0$-dimensional orthotope} is a set containing a single point.

A special subclass of orthotopes are the \emph{cubes}. For $d\in\N$ with $d\le n$, a \defn{$d$-dimensional cube} is a $d$-dimensional orthotope isometric to a \defn{standard $d$-dimensional cube}, i.e., a standard orthotope of the form $[-a,a]^d$, $a\in(0,+\infty)$. 
Note that $C\subseteq\R^n$ is a $d$-dimensional cube if and only if $C$ is mapped by a conformal affine map (i.e.,~a map in the form $\lambda \id\circ E$, where $E\in\Isom(\R^n)$ and $\lambda>0$) to $\stdCube^d\=[-1,1]^d$. A \defn{$0$-dimensional cube} is a set containing a single point. 

For a standard orthotope $Q\=\prod_{i=1}^d[-a_i,a_i]$, we call the linear transformation
\begin{equation}\label{e: ort Latt: cubic stretching}
\pi\:\R^n\to\R^n, \quad (x_1,\,\dots,\,x_d,\,x_{d+1},\,\dots,\,x_n)\mapsto(x_1/a_1,\,\dots,\,x_d/a_d,\,x_{d+1},\,\dots,\,x_n)
\end{equation}
the \defn{cubic-stretching} of $Q$. Clearly $\pi(Q)=\stdCube^d$.

The \defn{orthotopic structure} of a standard orthotope $Q=\prod_{i=1}^d[-a_i,a_i]$ is the cell decomposition
\begin{equation*}
\Rec_d(Q)\=\prod_{i=1}^d\{[-a_i,a_i],\,\{-a_i\},\,\{a_i\}\}.
\end{equation*}
For a $d$-dimensional orthotope $R\subseteq\R^n$, the orthotopic structure is given by the pullback (cf.~Lemma~\ref{l: cellular: cellular homeo}) $\Rec_d(R)\=T^*\Rec_d(Q)$, where $T\in\Isom(\R^n)$ maps $R$ to a standard $d$-dimensional orthotope $Q$. Note that $\Rec_d(R)$ does not depend on the choice of $T$.
Furthermore, we denote
$\partial\Rec_d(R)\=\Rec_d(R)\smallsetminus\{R\}=\Rec_d(R)|_{\cellbound R}$.
It is easy to see that
$\partial \Rec_n(R)$ is a cell decomposition of $\cellbound R$.
We call an element in $\Rec_d(R)$ a \defn{face of $R$}, and call a face of dimension $d-1$ a \defn{facet}. 

In particular, we define the \defn{cubic structure} of $\stdCube^d$ by 
$
\Cube_d\=\Rec_d\bigl(\stdCube^d\bigr)=\{[-1,1],\,\{-1\},\,\{1\}\}^d,
$
and the cubic structure of a $d$-dimensional cube $C$ by $\Cube_d(C)\=L^*\Cube_d$, where $L$ is a conformal affine map that maps $C$ to $\stdCube^d$. We also denote $\partial\Cube_d(C)\=\Cube_d(C)\smallsetminus\{C\}$ and $\partial\Cube_d\=\partial\Rec_d\bigl(\stdCube^d\bigr)$.
Clearly, for a standard $d$-dimensional orthotope $Q$, the orthotopic structure $\Rec_d(Q)$ is identical to the pullback $\pi^*\Cube_d$ by the cubic stretching $\pi$ of $Q$.

In addition, for a $0$-dimensional orthotope (quivalently, $0$-dimensional cube) $P$, we define $\Rec_0(P)=\Cube_0(P)\=\{P\}$ and $\partial\Rec_0(P)=\partial\Cube_0(P)\=\emptyset$.

In \cite{Yu13}, a crystallographic group $\Gamma$ on $\R^n$ is called \defn{cubic-type} if normal fundamental domains of $\Gamma$ can be realized as $n$-dimensional cubes. 

Here we introduce a similar notion called \defn{orthotopic crystallographic group}. We call a crystallographic group $\Gamma$ \defn{orthotopic} if there exists an $n$-dimensional orthotope that is a normal fundamental domain of $\Gamma$. Examples of orthotopic crystallographic groups include $\Gamma_{\operatorname{tor}}\cong\Z^n$ and $\Gamma_{\operatorname{sph}}\cong \Z^n\rtimes\Z_2$ whose quotient spaces are $\mathbb{T}^n$ and $\Sp^n$, respectively.

With orthotopic crystallographic groups introduced, we define a special class of Latt\`es maps. Recall the definition of \emph{Latt\`es triples} and \emph{Latt\`es maps} in Definition~\ref{d: Lattes triple}.
\begin{definition}[Orthotopic Latt\`es maps]
	A Latt\`es triple $(\Gamma,h,A)$ is \defn{orthotopic} if $\Gamma$ is an orthotopic crystallographic group with a normal fundamental domain $Q=\prod_{i=1}^n[0,a_i]$, and $A=\lambda I$ where $\lambda\in\N$. An \defn{orthotopic Latt\`es map} is a Latt\`es map with respect to an orthotopic Latt\`es triple.
\end{definition}

For examples of orthotopic Latt\`es maps, see e.g.~Astola, Kangaslampi, and Peltonen \cite{AKP10}.

\subsection{Symmetry group and symmetric decomposition of cube} 
As in the previous subsection, we fix an arbitrary $n\in\N$ and use the convention that all discussions take place in the ambient space $\R^n$. We also identify $\R^d$, $0\le d\leq n$, with the subspace $\R^d\times\{0\}^{n-d}$ of $\R^n$.

We study an orthotopic crystallographic group $\Gamma<\Isom(\R^n)$ by decomposing its fundamental domain into some smaller polyhedra. We call such a decomposition a \emph{symmetric decomposition}. Then in Subsection~\ref{subsct: orthotopic Lattes are MKV}, we use symmetric decomposition to construct cellular Markov partitions for orthotopic Latt\`es maps.

First, we consider fundamental domains of cubic type. Then we naturally proceed to the orthotopic case, extending the results for cubic-type cases.

In preparation, we recall the notion of (cubic) symmetry.
For the $d$-dimensional cube $\stdCube^d\=[-1,1]^d\subseteq\R^d$ (where $\R^d$ is identified as the subspace $\R^d\times\{0\}^{n-d}$ of $\R^n$), a \defn{symmetry of $\stdCube^d$ (or $\stdCube^d$-symmetry)} is an isometry $\gamma\in \Isom(\R^n)$ with $\gamma\bigl(\stdCube^d\bigr)=\stdCube^d$. All symmetries of $\stdCube^d$ form a subgroup $\Sym\bigl(\stdCube^d\bigr)<\Isom(\R^n)$
called the \defn{symmetry group of $\stdCube^d$}. 

For each $i\in\{1,\,\dots,\,n\}$, denote by $e_i$ the point in $\R^n$ for which the $i$-th coordinate is $1$ and the others are $0$. Usually we do not distinguish the points $e_i$, $i\in\{1,\,\dots,\,n\}$, and the canonical orthogonal base $e_1,\,\dots,\,e_n$.

\begin{rem}\label{r: ort Latt: sym(I) = DE}
    As an elementary fact, for each $\gamma\in \Sym\bigl(\stdCube^d\bigr)$, we have $\gamma\bigl(\R^d\bigr)=\R^d$, and the restriction $\gamma|_{\R^d}$
    is an orthogonal linear map having the form $\gamma|_{\R^d}=DE$ where $D,\,E\in\operatorname{O}(d)$ satisfy $De_i\in\{-e_i,\,e_i\}$ for $1 \leq  i \leq  d$, and the restriction of $E$ on the base $e_1,\,\dots,\,e_d$ is a permutation, that is, there exists $\sigma\in S_d$ such that $Ee_i=e_{\sigma(i)}$, $1\le i\le d$.  
\end{rem}

In what follows, we introduce symmetric decompositions. Briefly speaking, we subdivide a cube using half-spaces; cf.~\cite[Section~5.3]{SV93} for a similar discussion.

Fix arbitrary $d \in\N$ with $d\leq n$. For all $i,\,j\in\{1,\,\dots,\,d\}$ and $a,\,b\in\{-1,\,1\}$, we denote
\begin{equation}\label{e: ort Latt: Hi and Hij}
	H^d_{ij}(a,b)\=\bigl\{(x_1,\,\dots,\,x_d)\in \stdCube^d:ax_i \geq  bx_j\bigr\}\subseteq\R^d.
\end{equation}
Note that $H^d_{ii}(a,b)=\bigl\{(x_1,\,\dots,\,x_d)\in \stdCube^d:(a-b)x_i \geq 0\bigr\}$ and $H^d_{ij}(a,b)=H^d_{ji}(-b,-a)$.
Set 
\begin{equation}\label{e: ort Latt: cH}
\cH^d\=\bigl\{H^d_{ij}(a,b):a,\,b\in\{-1,1\},\,1 \leq  i,\,j \leq  d\bigr\}.
\end{equation}
We call $\cH_*\subseteq\cH^d$ a \defn{{\fundSet}} of $\cH^d$ if, for all $1 \leq  i\leq j \leq  d$,
\begin{equation}\label{e: ort Latt: fund subset}
    \cH_*\cap\bigl\{H^d_{ij}(1,-1),\,H^d_{ij}(-1,1)\bigr\}\neq\emptyset\quad\text{and}\quad\cH_*\cap\bigl\{H^d_{ij}(1,1),\,H^d_{ij}(-1,-1)\bigr\}\neq\emptyset.
\end{equation}

Denote
\begin{equation}\label{e: ort Latt: definition Ko}
\Ko_d\=\bigl\{\bigcap\cH_*:\cH_*\text{ is a {\fundSet} of }\cH^d\bigr\}.
\end{equation}
For each $d$-dimensional cube $C\subseteq\R^n$ we denote $\Ko_d(C)\=L^*\Ko_d\=\bigl\{L^{-1}(H):H\in\Ko_d\bigr\}$, where $L$ is a conformal affine map with $L(C)=\stdCube^d$. In addition, we denote  $\Ko_0\=\{\{0\}\}$.

\begin{remark}
Note that $\Ko_d(C)$ is well defined, i.e., independent of the choice of $L$ (see Lemma~\ref{l: orth Latt: Ko is invariant}~(i)).
\end{remark}

Now we define symmetric decompositions of cubes. For a $d$-dimensional cube $C$ with $d\in\N$, we call
\begin{equation}\label{e: orth Latt: symmetric decomp}
	\cK_d(C)\=\bigcup_{c\in\Cube_d(C)}\Ko_{\dim(c)}(c)
    \quad\text{ and }\quad
    \partial\cK_d(C)\=\bigcup_{c\in\partial\Cube_d(C)}\Ko_{\dim(c)}(c)
\end{equation}
the \defn{symmetric decomposition} of $C$ and $\cellbound C$, respectively. In particular, we denote $\cK_d\=\cK_d\bigl(\stdCube^d\bigr)$ and $\partial\cK_d\=\partial\cK_d\bigl(\stdCube^d\bigr)$ for convenience. Moreover, define $\cK_0\=\{\{0\}\}$ and $\partial\cK_0=\emptyset$.
\begin{remark}
    Given a conformal affine map $L$ mapping $C$ to $\stdCube^d$, we have $\cK_d(C)=L^*\cK_d\=\{L^{-1}(H):H\in\cK_d\}$ (see Lemma~\ref{l: orth Latt: Ko is invariant}~(ii)).
\end{remark} 

\begin{figure}[h]
	\includegraphics[scale=0.65]{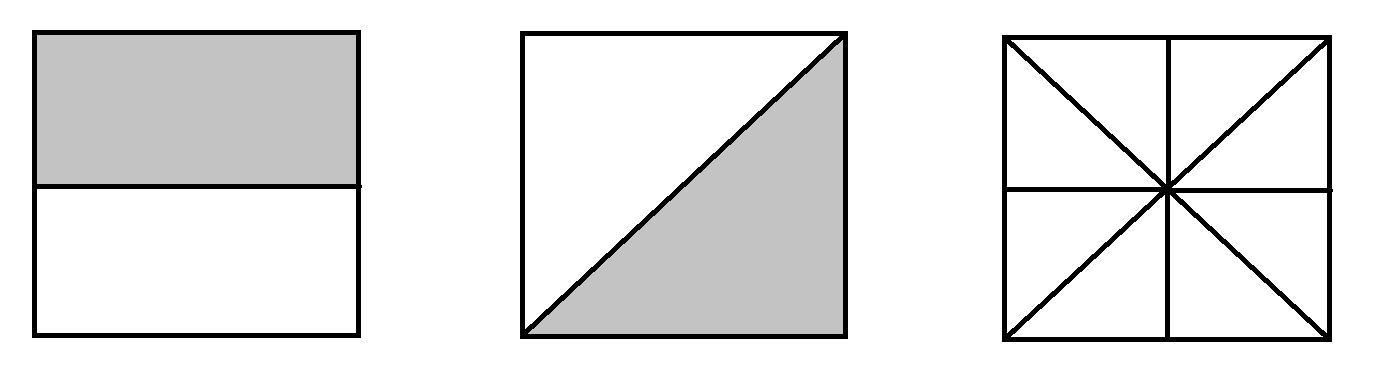}
	\caption{$H^{2}_{22}(1,-1)$, $H^2_{12}(1,1)$, and $\cK_{2}$.}
\end{figure}

The following lemma guarantees that symmetric decompositions are well defined.
\begin{lemma}\label{l: orth Latt: Ko is invariant}
    Let $d\in\N_0$ with $d\leq n$ be arbitrary. Then the following statements are true:
    \begin{enumerate}[label=(\roman*),font=\rm]
        \smallskip
        \item For each $T\in\Sym\bigl(\stdCube^d\bigr)$, $\Ko_d=T^*\Ko_d$. Moreover, for each $d$-dimensional cube $C$ and each pair of conformal affine maps $L,\,L'$ mapping $C$ onto $\stdCube^d$, we have $L^*\Ko_d=(L')^*\Ko_d=\Ko_d(C)$.
        \smallskip
        \item For each pair $C,\,C'$ of $d$-dimensional cubes and each conformal affine map $L$ mapping $C$ to $C'$, we have $\Ko_d(C)=L^*\Ko_d(C')$, $\cK_d(C)=L^*\cK_d(C')$, and $\partial\cK_d(C)=L^*\partial\cK_d(C')$.
    \end{enumerate}
\end{lemma}
\begin{proof} 
	Clearly, the statements are true if $d=0$. Now suppose $d\in\N$.
	
	\smallskip
	
    (i) Since $T,\,T^{-1}\in\Sym\bigl(\stdCube^d\bigr)$, by Remark~\ref{r: ort Latt: sym(I) = DE}, we may assume $T^{-1}e_i=\delta_i e_{\sigma(i)}$, $\delta_i\in\{-1,\,1\}$, for each $1 \leq  i \leq  d$, where $\sigma\in S_d$ is a permutation. Then for all $i,\,j\in\{1,\,\dots,\,d\}$ and $a,\,b\in\{-1,\,1\}$,
    \begin{equation*}
    T^{-1}\bigl(H^d_{ij}(a,b)\bigr)=\bigl\{(x_1,\,\dots,\,x_d)\in \stdCube^d:ax_{\sigma(i)}/\delta_i \geq  bx_{\sigma(j)}/\delta_j\bigr\}
    =H^d_{\sigma(i)\sigma(j)}(a/\delta_i,b/\delta_j),
    \end{equation*}  
    which implies that for each {\fundSet} $\cH_*\subseteq\cH^d$ (cf.~(\ref{e: ort Latt: fund subset}) for definition), $T^*\cH_*\=\bigl\{T^{-1}(H_*):H_*\in\cH_*\bigr\}$ is a {\fundSet}. Thus, by the definition of $\Ko_d$ (cf.~(\ref{e: ort Latt: definition Ko})), for each $H\in\Ko_d$, $T^{-1}(H)\in\Ko_d$. It follows that $T^*\Ko_d\subseteq\Ko_d$. Likewise, $\Ko_d\subseteq T^*\Ko_d$. 

    Let $C$ be a $d$-dimensional cube and $L,\,L'$ be conformal affine maps that map $C$ to $\stdCube^d$. It is easy to see that $L'\circ L^{-1}\in\Sym\bigl(\stdCube^d\bigr)$, and thus, as shown above, $\bigl(L'\circ L^{-1}\bigr)^*\Ko_d=\Ko_d$. So $L^*\Ko_d=(L')^*\Ko_d=\Ko_d(C)$.

    \smallskip

    (ii) Let $C,\,C'$ be $d$-dimensional cubes and $L$ a conformal affine map with $L(C)=C'$. Let $P$ and $P'$ be conformal affine maps mapping $C$ and $C'$, respectively, to $\stdCube^d$. It is clear that $P'\circ L \circ P^{-1}\in\Sym\bigl(\stdCube^d\bigr)$. Thus, by (i), $\Ko_d(C)=P^*\Ko_d=P^*\bigl(P'\circ L \circ P^{-1}\bigr)^*\Ko_d=L^* (P')^*\Ko_d=L^*\Ko_d(C')$.
    
    Now consider $\partial\cK_d(C)$ and $\cK_d(C)$. 
    It is easy to verify $\Cube_d(C)=L^*\Cube_d(C')$ and $\partial\Cube_d(C)=L^*\partial\Cube_d(C')$. Combining this with the fact that $\Ko_d(C)=L^*\Ko_d(C')$ for arbitrary $d\in\N$ with $d\le n$ (as shown above) and the definition of symmetric decomposition in (\ref{e: orth Latt: symmetric decomp}), we obtain $\cK_d(C)=L^*\cK_d(C')$ and $\partial\cK_d(C)=L^*\partial\cK_d(C')$.
\end{proof}

We further discuss the structure of $\Ko_n$, giving some elementary properties.
\begin{lemma}\label{l: ort Latt: Ko Union and Intersect}
    The following statements are true:
    \begin{enumerate}[font=\rm,label=(\roman*)]
        \smallskip
        \item $\stdCube^n=\bigcup\Ko_n$ and $H_0=\bigcup\{\bigcap\cH_0:\cH_0\subseteq\cH^n\text{ is a {\fundSet} with } H_0\in\cH_0\}$ for each $H_0\in\cH^n$.
        \smallskip
        \item $H_1\cap H_2\in\Ko_n$ for all $H_1,\,H_2\in\Ko_n$.
    \end{enumerate} 
\end{lemma}
\begin{proof}
 (i) Let $x\=(x_1,\,\dots,\,x_n)\in \stdCube^n$ be arbitrary. By~(\ref{e: ort Latt: Hi and Hij}), for all $1\leq i\leq j\leq n$ and $\delta\in\{-1,\,1\}$, we can choose $a,\,b\in\{-1,\,1\}$ such that $ab=\delta$ and $x\in H^n_{ij}(a,b)$. This gives a {\fundSet} (cf.~(\ref{e: ort Latt: fund subset}) for definition) $\cH_*\subseteq\cH^n$ for which $x\in\bigcap\cH_*\in\Ko_n$. Thus, $\stdCube^n\subseteq\bigcup\Ko_n$. Clearly $\bigcup\Ko_n\subseteq \stdCube^n$.

Furthermore, fix arbitrary $H_0\in\cH^n$ and $x_0\in H_0$. As discussed above, there exists a {\fundSet} $\cH'_*\subseteq\cH^n$ for which $x_0\in\bigcap\cH'_*$. Then $\cH'_0\=\cH'_*\cup\{H_0\}$ is a {\fundSet} for which $x_0\in\bigcap\cH'_0$. It follows that $H_0\subseteq\bigcup\{\bigcap\cH_0:\cH_0\subseteq\cH^n\text{ is a {\fundSet} with } H_0\in\cH'_*\}$. The converse inclusion is clear.

\smallskip

(ii) Consider $H_1,\,H_2\in\Ko_n$. Suppose $H_1=\bigcap\cH_1$ and $H_2=\bigcap\cH_2$, where $\cH_1,\,\cH_2\subseteq\cH^n$ are {\fundSet}s. By the definition of {\fundSet}s (cf.~(\ref{e: ort Latt: fund subset})), $\cH_1\cup\cH_2$ is a {\fundSet}, and thus $H_1\cap H_2=\bigcap(\cH_1\cup\cH_2)\in\Ko_n$.
\end{proof}

\begin{figure}[h]
    \centering
    \includegraphics[width=0.65\linewidth]{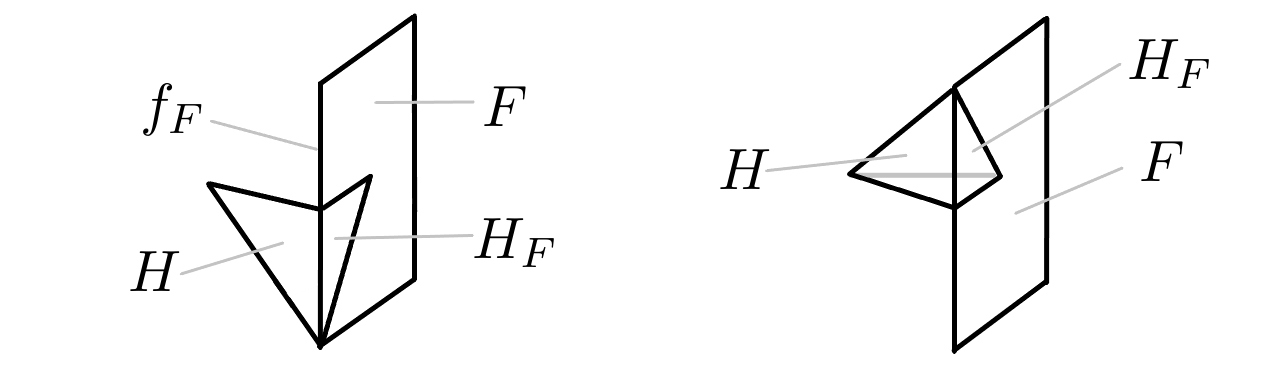}
    \caption{Illustration of Lemmas~\ref{l: orthotopic Lattes: property of K^o_n} (left) and~\ref{l: orthotopic Lattes: K_n-1 in dK_n} (right).}
\end{figure}

\begin{lemma}\label{l: orthotopic Lattes: property of K^o_n}
    Fix $H\in\Ko_n$. Then the following statements are true:
    \begin{enumerate}[label=(\roman*),font=\rm]
    	\smallskip
        \item There exists $T\in\Sym(\stdCube^n)$ such that $T(H)\subseteq[0,1]^n$.
        \smallskip
        \item If $H\cap\cellbound \stdCube^n\neq\emptyset$, then there exists a facet $F\in\Cube_n^{[n-1]}$ such that $H\cap\cellbound \stdCube^n\subseteq F$.
        \smallskip
        \item If $F\in \Cube_n^{[n-1]}$ satisfies $H\cap F\neq\emptyset$, then there exists $H_F\in\Ko_{n-1}(F)$ and $f_F\in \Cube_{n-1}(F)$ such that $H\cap F= H_F\cap f_F$.
    \end{enumerate}
\end{lemma}
\begin{proof}
    By the definition of $\Ko_n$ (cf.~(\ref{e: ort Latt: definition Ko})), suppose $H=\bigcap\cH_*$, where $\cH_*\subseteq\cH^n$ is a {\fundSet} (cf.~(\ref{e: ort Latt: fund subset}) for definition).

    \smallskip

    (i) By the definition of {\fundSet}s, for each $i\in\{1,\,\dots,\,n\}$, pick $H^n_{ii}(a_i,-a_i)\in\cH_*$, where $a_i\in\{-1,\,1\}$. Let $T\in\operatorname{O}(n)$ be such that $Te_i=a_ie_i$, $i\in\{1,\,\dots,\,n\}$. It can be checked that $T\in\Sym(\stdCube^n)$ and $T(H)\subseteq\bigcap_{i=1}^n T(H_{ii}^n(a_i,-a_i))\subseteq[0,1]^n$.

    \smallskip

    (ii) Since $\Ko_n$ is invariant under $\stdCube^n$-symmetries (Lemma~\ref{l: orth Latt: Ko is invariant}~(i)), we may assume by~(i) that $H\subseteq[0,1]^n$. 
    
	By the definition of {\fundSet}s, for all $1 \leq  i < j \leq  n$, we can choose $H^{n}_{ij}(a_{ij},a_{ij})\in\cH_*$, where $a_{ij}\in\{-1,\,1\}$. 

	We construct a directed graph $G=(V,E)$ with $n$ vertices $V=\{v_1,\,\dots,\,v_n\}$ by defining the set of edges $E=\{e_{ij}:1 \leq  i<j \leq  n\}$ as follows: for $1 \leq  i<j \leq  n$, if $a_{ij}=1$, then set $e_{ij}\=(v_i,v_j)$; otherwise, when $a_{ij}=-1$, set $e_{ij}\=(v_j,v_i)$.
	Using simple induction arguments, we can find an integer $1\le m\le n$ such that for each $v_i\in V$, there exists a path from $v_m$ to $v_i$. By the construction of $G$ and (\ref{e: ort Latt: Hi and Hij}), each $(x_1,\,\dots,\,x_n)\in \bigcap_{1\le i<j\le n}H^n_{ij}(a_{ij},a_{ij})$ satisfies $x_m \geq  x_i$, $1 \leq  i \leq  n$.
	 
	Let $x=(x_1,\,\dots,\,x_n)\in H\cap\cellbound \stdCube^n$ be arbitrary. By the assumption $H\subseteq[0,1]^n$, there exists an integer $1 \leq  k \leq  n$ for which $x_k=1$. Since $x\in H=\bigcap\cH_* \subseteq\bigcap_{1\le i<j\le n}H^n_{ij}(a_{ij},a_{ij})$, we have, as shown above, $x_m\geq x_k=1$. Hence, $H\cap\cellbound \stdCube^n$ is contained in the facet $\{(x_1,\,\dots,\,x_n)\in \stdCube^n:x_m=1\}$.

     \smallskip

     (iii) Since $\Ko_n$ is invariant under $\stdCube^n$-symmetries (see~Lemma~\ref{l: orth Latt: Ko is invariant}~(i)), we may assume that $H$ intersects the facet $F\=\{(x_1,\,\dots,\,x_n)\in \stdCube^n:x_n=1\}$.
     Consider the translation $L\:(x_1,\,\dots,\,x_{n-1},\,1)\mapsto(x_1,\,\dots,\,x_{n-1},\,0)$ mapping $F$ to $\stdCube^{n-1}$.

    Set $\cH'_*\=\bigl\{H^{n-1}_{ij}(a,b)\in\cH^{n-1}:H^{n}_{ij}(a,b)\in\cH_*,\, a,\,b\in\{-1,\,1\},\, 1\le i\le j\le n-1\bigr\}$ and $\cF\=\bigl\{H^{n}_{ij}(a,b)\cap F:H^{n}_{ij}(a,b)\in\cH_*,\,a,\,b\in\{-1,\,1\},\,n\in\{i,\,j\}\bigr\}$. In what follows, we show $H_F\=L^{-1}(\bigcap\cH'_*)$ and $f_F\=\bigcap \cF$ satisfy the desired property.

    Suppose $H^n_{ij}(a,b)\in\cH_*$, where $1\leq i\leq j\leq n$ and $a,\,b\in\{-1,\,1\}$. Since $H\cap F=F\cap\bigcap\cH_*\neq\emptyset$, we have $H^n_{ij}(a,b)\cap F\neq\emptyset$.
     If $i\neq n$ and $j\neq n$, then 
     \begin{equation*}
         H^{n}_{ij}(a,b)\cap F=\{(x_1,\,\dots,\,x_n)\in \stdCube^n:ax_i \geq  bx_j,\,x_n=1\}=L^{-1}\bigl(H_{ij}^{n-1}(a,b)\bigr).
     \end{equation*}
     If $i=n$, then the non-empty set $H_{ij}^{n}(a,b)\cap F=\{(x_1,\,\dots,\,x_n)\in \stdCube^n:a \geq bx_j,\,x_n=1\}$ is either an $(n-2)$-dimensional face of $F$, or equal to $F$;
     the same holds for the case $j=n$. 
     
 The above discussion implies that $\cF$ is a collection of faces of $F$ and $\{H_*\cap F : H_*\in\cH_*\}=\bigl\{L^{-1}(H'_*):H'_*\in\cH'_*\bigr\}\cup\cF$. Thus, $H\cap F=\bigcap\cH_*\cap F=L^{-1}(\bigcap\cH'_*)\cap\bigcap\cF=H_f\cap f_F$.
 
 The construction of $\cH'_*$ implies that $\cH'_*$ is a {\fundSet} of $\cH^{n-1}$. Thus, $\bigcap\cH'_*\in\Ko_{n-1}$ and $H_F=L^{-1}(\bigcap\cH'_*)\in L^*\Ko_{n-1}=\Ko_{n-1}(F)$.
 On the other hand, the assumption $H\cap F\neq\emptyset$ implies that $f_F=\bigcap\cF$ is non-empty, and hence, being the intersection of a collection of faces, is itself a face of $F$.
\end{proof}

\begin{lemma}\label{l: orthotopic Lattes: K_n-1 in dK_n}
	For each facet $F\in\Cube_n^{[n-1]}$ and each $H_F\in\Ko_{n-1}(F)$, there exists $H\in\Ko_n$ for which $H_F=H\cap F=H\cap\cellbound \stdCube^n$.
\end{lemma}
\begin{proof}
	Since $\Ko_n$ is invariant under $\stdCube^n$-symmetries (see~Lemma~\ref{l: orth Latt: Ko is invariant}~(i)), we may assume $F\=\{(x_1,\,\dots,\,x_n)\in \stdCube^n:x_n=1\}$. Consider the translation 
	$L\:(x_1,\dots,x_{n-1},1)\mapsto (x_1,\dots,x_{n-1},\,0)$.
	
	Fix $H_F\in\Ko_{n-1}(F)=L^*\Ko_{n-1}$. By the definition of $\Ko_{n-1}$ (cf.~(\ref{e: ort Latt: definition Ko})), suppose $H_F=L^{-1}(\bigcap\cH_F)$, where $\cH_F\subseteq\cH^{n-1}$ is a {\fundSet} of $\cH^{n-1}$.
	It can be directly checked that 
    \begin{equation*}
    \cH_*\=\{A\times [-1,1]:A\in\cH_F\}\cup\{H^n_{in}(-1,-1),\,H^n_{in}(1,-1):1\le i\le n\}
    \end{equation*}
    is a {\fundSet} of $\cH^{n}$ such that $H\=\bigcap\cH_*\in\Ko_n$ satisfies $H_F=H\cap F=H\cap\cellbound \stdCube^n$.
\end{proof}

The above two lemmas yield a description of the structure of $\partial\cK_n$.

\begin{lemma}\label{l: orthotopic Lattes: structure of dK_n}
	For each $d$-dimensional cube $C$, $d\in\N$ with $d\leq n$, we have
	\begin{equation}\label{e: orthotopic Lattes: structure of dKo}
			\begin{aligned}
			\partial\cK_d(C)
            =
			&\{H\cap\cellbound C:H\in\Ko_d(C),\,H\cap \cellbound C\neq\emptyset\}\\
			=
			&\bigl\{H\cap F:H\in\Ko_d(C),\,F\in \Cube_d^{[d-1]}(C),\,H\cap F\neq\emptyset\bigr\}\\
			=
			&\{H\cap f:H\in\Ko_d(C),\,f\in \partial \Cube_d(C),\,H\cap f\neq\emptyset\}.
		\end{aligned}
	\end{equation}
\end{lemma}
\begin{proof}
We argue by induction on the dimension $n$ of the ambient space $\R^n$. Clearly, the lemma is true when $n=1$. Now suppose $n\geq 2$ and make the induction hypothesis that the lemma holds when the dimension of the ambient space is in $\{1,\,\dots,\,n-1\}$, or equivalently, (\ref{e: orthotopic Lattes: structure of dKo}) holds for all cubes of dimension $1\le d\leq n-1$.

Now consider an arbitrary $n$-dimensional cube $C$. By mapping $C$ to the standard $n$-dimensional cube $\stdCube^n$ using a conformal affine map, it suffices to assume $C={\stdCube}^n$ and consider $\partial\cK_n$.

First, by Lemma~\ref{l: orthotopic Lattes: property of K^o_n}~(ii), we have
\begin{equation*}
	\begin{aligned}
		\{H\cap\cellbound \stdCube^n:H\in\Ko_n,\,H\cap \cellbound \stdCube^n\neq\emptyset\}
		\subseteq
		&\bigl\{H\cap F:H\in\Ko_n,\,F\in \Cube_n^{[n-1]},\,H\cap F\neq\emptyset \bigr\}\\
		\subseteq
		&\{H\cap f:H\in\Ko_n,\,f\in \partial \Cube_n,\,H\cap f\neq\emptyset\}.
	\end{aligned}
\end{equation*}

	Then we prove 
	$\{H\cap f:H\in\Ko_n,\,f\in \partial \Cube_n,\,H\cap f\neq\emptyset\}
	\subseteq\partial\cK_n$.
	Let $H\in\Ko_n$ and $f\in \partial \Cube_n$ satisfy $H\cap f\neq\emptyset$. Let $F$ be a facet for which $f\subseteq F$. By Lemma~\ref{l: orthotopic Lattes: property of K^o_n}~(iii), there exist $H_F\in\Ko_{n-1}(F)$ and $f_F\in \Cube_{n-1}(F)$ for which $H\cap F=H_F\cap f_F$. Thus, since $f\subseteq F$, we have $H\cap f=H_F\cap(f\cap f_F)$. Clearly $f\cap f_F\in \Cube_{n-1}(F)$. If $f\cap f_F=F$, then $H\cap f=H_F\in\Ko_{n-1}(F)$. Otherwise, when $f\cap f_F\neq F$,  by the induction hypothesis, $H\cap f=H_F\cap (f\cap f_F)\in\partial\cK_{n-1}(F)$. In both cases,
    $H\cap f\in\cK_{n-1}(F)\subseteq\partial\cK_{n}$. 

	Finally, we prove $\partial\cK_n\subseteq\{H\cap\cellbound \stdCube^n:H\in\Ko_n,\,H\cap \cellbound \stdCube^n\neq\emptyset\}$.
	
    Fix $H'\in\partial\cK_n$. By the definition of $\partial\cK_n$ (cf.~(\ref{e: orth Latt: symmetric decomp})), suppose $H'\in\Ko_d(f)$ for some $0 \leq  d \leq  n-1$ and $f\in \Cube_n^{[d]}$. If $d=n-1$, then by Lemma~\ref{l: orthotopic Lattes: K_n-1 in dK_n}, $H'\in\{H\cap\cellbound \stdCube^n:H\in\Ko_n,\,H\cap \cellbound \stdCube^n\neq\emptyset\}$.
	
    Now suppose $d<n-1$. Let $F_i\in\Cube_n^{[n-1]}$, $i\in\cI$, be all the facets of $\stdCube^n$ containing $f$. 
    
    Fix $j\in\cI$. Clearly $f\in\partial\Cube_{n-1}(F_j)$ and $H'\in\partial\cK_{n-1}(F_j)$. By the induction hypothesis, we may choose $H'_j\in\Ko_{n-1}(F_j)$ with $H'=H'_j\cap\cellbound F_j$. By Lemma~\ref{l: orthotopic Lattes: K_n-1 in dK_n}, there is $H_j\in\Ko_n$ such that $H'_j=H_j\cap F_j=H_j\cap\cellbound \stdCube^n$. Thus, $H'=H_j\cap \cellbound F_j$.

    It is easy to check $f=\bigcap_{i\in\cI} F_i=\bigcap_{i\in\cI} \cellbound F_i$. This, combined with the fact that $H_i\cap F_i=H_i\cap\cellbound \stdCube^n$ and $H'=H_i\cap \cellbound F_i$ for each $i\in\cI$, implies
    $H'=\bigcap_{i\in\cI} H_i\cap\cellbound F_i=\bigcap_{i\in\cI}H_i\cap F_i=\bigcap_{i\in\cI} H_i\cap\cellbound \stdCube^n$.
	Thus $H'\in \{H\cap\cellbound \stdCube^n:H\in\Ko_n,\,H\cap \cellbound \stdCube^n\neq\emptyset\}$ since $\bigcap_{i\in\cI}H_i\in\Ko_n$ (see~Lemma~\ref{l: ort Latt: Ko Union and Intersect}~(ii)).    
\end{proof}


As a consequence of Lemma~\ref{l: orthotopic Lattes: structure of dK_n}, $\Ko_n$ is the collection of cones whose bases are in $\partial\cK_n$.
Recall that for a point $x\in\R^n$ and a subset $Y\subseteq\R^n$, the cone with tip $x$ and base $Y$ is
\begin{equation*}
\Cone(x,Y)\=\{x+\lambda(y-x):y\in Y,\,\lambda\in[0,1]\}.
\end{equation*}

\begin{cor}\label{c: ort Latt: each cell is Cone(0,H')}
	For each $d\in\N$ with $d\leq n$, we have $\Ko_d=\{\{0\}\}\cup\{\Cone(0,H'):H'\in\partial\cK_d\}$. Moreover, each element in $\cK_d$ is a simplex.
\end{cor}
Recall that a $d$-simplex ($d\leq n$) in $\R^n$ is the convex hull of $d+1$ affinely independent points.
\begin{proof}
    It is easy to see from the definition of $\Ko_d$ (see~(\ref{e: ort Latt: Hi and Hij}) and (\ref{e: ort Latt: definition Ko})) that $\{0\}\in\Ko_n$ and for each $H\in\Ko_d\smallsetminus\{\{0\}\}$, $H=\Cone\bigl(0,H\cap\cellbound \stdCube^d\bigr)$, and thus, by Lemma~\ref{l: orthotopic Lattes: structure of dK_n}, $H=\Cone(0,H')$ for some $H'\in\partial\cK_d$. Conversely, for each $H'\in\partial\cK_d$, by Lemma~\ref{l: orthotopic Lattes: structure of dK_n}, there exists $H\in\Ko_d$ such that $H'=H\cap\cellbound \stdCube^d$, and thus, as shown above, $\Cone(0,H')=\Cone\bigl(0,H\cap\cellbound \stdCube^d\bigr)=H$. So $\Ko_d=\{\Cone(0,H'):H'\in\partial\cK_d\}\cup\{\{0\}\}$. 
    
    It follows immediately that for each $d$-dimensional cube $C$, elements in $\Ko_d(C)$ are either singletons, or cones with bases in $\partial\cK_d(C)$. Thus, by (\ref{e: orth Latt: symmetric decomp}) and since $\cK_0=\{\{0\}\}$, we can inductively conclude that each $H\in\cK_d$ is a simplex.
\end{proof}

Since an element of $\cK_n$ is a simplex, it is a topological cell. Moreover, we show that $\cK_n$ is a cell decomposition of $\stdCube^n$.

\begin{prop}\label{p: ort Latt: K_n is cell decomp}
	$\cK_n$ is a cell decomposition of $\stdCube^n$ that refines $\Cube_n$.
\end{prop}
\begin{proof}
	We prove the claim by induction. Clearly $\cK_0=\{\{0\}\}$ and $\cK_1=\{\{-1\},\,\{0\},\,\{1\},\,[-1,0],\,[0,1]\}$
	are cell decompositions refining $\Cube_0$ and $\Cube_1$, respectively. 
	Fix $n\in\N$, $n\geq2$, and make the induction hypothesis that for each $0 \leq  d \leq  n-1$, $\cK_d$
	 is a cell decomposition that refines $\Cube_d$ (equivalently, for each $d$-dimensional cube $C$, $\cK_d(C)$ is a cell decomposition of $C$ that refines $\Cube_d(C)$).
	
	First, by the induction hypothesis and the construction in (\ref{e: orth Latt: symmetric decomp}), we can use Lemma~\ref{l: construct refinement} to conclude that $\partial\cK_n$ is a cell decomposition that refines $\partial \Cube_n$. 
	
	Now, we verify conditions (i)--(iv) in Definition~\ref{d: cell decomposition} for $\cK_n$. 
	Conditions~(i) and (iv) hold since $\cK_n$ is a finite cover of $\stdCube^n$ (cf.~Lemma~\ref{l: ort Latt: Ko Union and Intersect}~(i)). It remains to verify (ii) and (iii).
	
	\smallskip
	
	(ii) Let $\sigma,\,\tau\in\cK_n$ be distinct. 
	If $\sigma,\,\tau\in\partial\cK_n$, then since $\partial\cK_n$ is a cell decomposition, $\brCellint{\sigma}\cap\brCellint{\tau}=\emptyset$.
	In what follows, we consider the case where $\sigma,\,\tau\in\Ko_n$. Clearly, $\stdCellint{\sigma}\cap\stdCellint{\tau}=\emptyset$ if $\sigma=\{0\}$ or $\tau=\{0\}$. Now suppose $\sigma\neq\{0\}$ and $\tau\neq\{0\}$. Then by Corollary~\ref{c: ort Latt: each cell is Cone(0,H')},
    there exist $\sigma_1,\,\tau_1\in\partial\cK_n$ such that $\sigma=\Cone(0,\sigma_1)$ and $\tau=\Cone(0,\tau_1)$. Since $\sigma\ne\tau$, we have $\sigma_1\ne\tau_1$, and thus (since $\partial\cK_n$ is a cell decomposition) $\brCellint{\sigma_1}\cap\brCellint{\tau_1}=\emptyset$. Then $\stdCellint{\sigma}=\{\lambda x:x\in\stdCellint{\sigma_1},\,\lambda\in(0,1)\}$ is disjoint from $\stdCellint{\tau}$.
	Finally, if $\sigma\in\Ko_n$ and $\tau\in\partial\cK_n$, it is clear that $\stdCellint{\sigma}\subseteq\stdCellint{\stdCube^n}$, and thus $\brCellint{\sigma}\cap\brCellint{\tau}=\emptyset$.
	
	\smallskip
	
	(iii) Fix $H\in\cK_n$. If $H\in\partial\cK_n$, then since $\partial\cK_n$ is a cell decomposition, $\cellbound H$ is a union of cells in $\partial\cK_{n}$. If $H=\{0\}$, then $\cellbound H=\emptyset$. Now suppose $H\in\Ko_n$, $H\neq\{0\}$. Then by Corollary~\ref{c: ort Latt: each cell is Cone(0,H')}, there exists $H_1\in\partial\cK_{n}$, for which $H=\Cone(0,H_1)=\{\lambda x:\lambda\in[0,1],\,x\in H_1\}$. Since $\partial\cK_n$ is a cell decomposition, $\cellbound H_1$ is a union of cells in $\partial\cK_{n}$, and thus (again by Corollary~\ref{c: ort Latt: each cell is Cone(0,H')}) $\Cone(0,\cellbound H_1)$ is a union of cells in $\Ko_n$. It follows that $\cellbound H= H_1\cup \Cone(0,\cellbound H_1)$ is a union of cells in $\cK_n$.

    Thus $\cK_n$ is a cell decomposition. By definition $\cK_n$ is a refinement of $\Cube_n$.
\end{proof}

Next, we consider a self-similar subdivision of $\cK_n$.
Fix $l\in\N$. For each $\alpha\=(\alpha_1,\,\dots,\,\alpha_n)\in\{0,\,\dots,\,l-1\}^n$, denote the $n$-dimensional cube centered at $x_\alpha\=\bigl(\frac{2\alpha_1+1-l}{l},\,\cdots,\,\frac{2\alpha_n+1-l}{l}\bigr)$ by
\begin{equation*}
C^n_{\alpha}\= [ (2\alpha_1-l)/l, (2\alpha_1-l+2)/l ]\times\cdots\times [ (2\alpha_n-l)/l , (2\alpha_n-l+2)/l ].
\end{equation*}
Using Lemma~\ref{l: construct refinement} and Proposition~\ref{p: ort Latt: K_n is cell decomp}, we construct a cell decomposition as follows:
\begin{equation}\label{e: ort Latt: K_n,l}
\cK_{n,l}\=\bigcup_{\alpha\in\{0,\,\dots,\,l-1\}^n}\cK_n(C^n_{\alpha}).
\end{equation}
For an $n$-dimensional cube $C$, we denote $\cK_{n,l}(C)\=L^*\cK_{n,l}$, where $L$ is a conformal affine map with $L(C)=\stdCube^n$. 
\begin{rem}\label{r: ort Latt: Kn,l is invariant}
    Note that $\cK_{n,l}(C)$ does not depend on the choice of $L$ since $\cK_{n,l}$ is invariant under each $T\in\Sym(\stdCube^n)$, which follows from Lemma~\ref{l: orth Latt: Ko is invariant}~(ii) and the observation that for each $\alpha\in\{0,\,\dots,\,l-1\}^n$, there exists $\beta\in\{0,\,\dots,\,l-1\}^n$ for which $TC^n_{\alpha}=C^n_\beta$.
\end{rem}

\begin{figure}[h]
    \centering
    \includegraphics[width=0.6\linewidth]{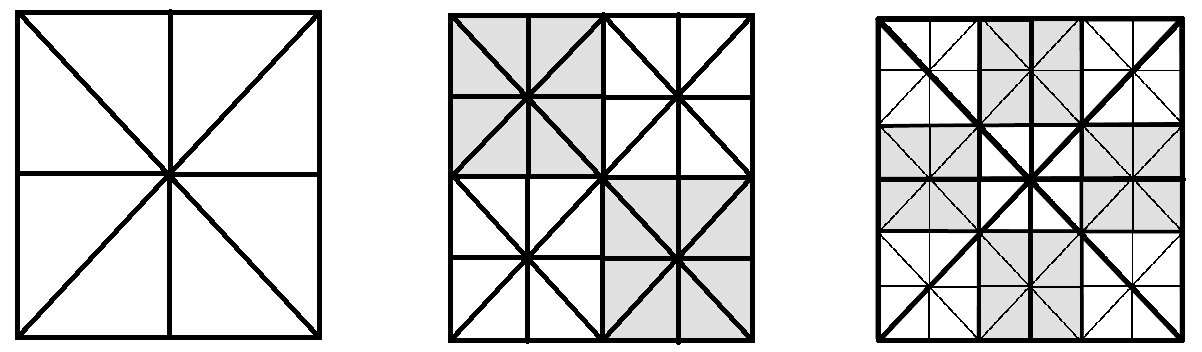}
	\caption{$\cK_2$, $\cK_{2,2}$, and $\cK_{2,3}$.}
\end{figure}

\begin{lemma} 
	For each $l\in\N$, the cell decomposition $\cK_{n,l}$ is a refinement of $\cK_n$.
\end{lemma}
\begin{proof}
	For each $\alpha\=(\alpha_1,\,\dots,\,\alpha_n)\in\{0,\,\cdots,\,l-1\}^n$, denote
	$x_\alpha\=\bigl(\frac{2\alpha_1+1-l}{l},\,\cdots,\,\frac{2\alpha_n+1-l}{l}\bigr)$
	and consider the conformal affine map
	$T_\alpha\:x\mapsto T_\alpha x\= l(x-x_\alpha)$
	that maps $C^n_\alpha$ to $\stdCube^n$ and $x_\alpha$ to $0$. 

    We verify conditions~(i) and (ii) in Definition~\ref{d: refinement}.
	
	\smallskip
	
	(i) Fix $\alpha\=(\alpha_1,\,\dots,\,\alpha_n)\in\{0,\,\cdots,\,l-1\}^n$. 
    Then for all $1 \leq  i,\,j \leq  n$ and $a,\,b\in\{-1,\,1\}$,
	\begin{equation*}
	T_{\alpha}\bigl(H^n_{ij}(a,b)\bigr) =\{(y_1,\,\dots,\,y_n)\in T_\alpha(\stdCube^n): ay_i-by_j\geq 2(b\alpha_j-a\alpha_i)+(b-a)(1-l)\}.
    \end{equation*}
    If $2(b\alpha_j-a\alpha_i)+(b-a)(1-l) \leq  0$, then $H^n_{ij}(a,b)\subseteq T_{\alpha}\bigl(H^n_{ij}(a,b)\bigr)$; see Figure~\ref{f: Knl is refinement} for an illustration. Otherwise, since $2(b\alpha_j-a\alpha_i)+(b-a)(1-l)\in 2\Z$, we have $H^n_{ij}(a,b)\subseteq[-1,1]^n\subseteq T_{\alpha}\bigl(H^n_{ij}(-a,-b)\bigr)$.

    By the above discussion, either $T_\alpha^{-1}\bigl(H^n_{ij}(a,b)\bigr)\subseteq H^n_{ij}(a,b)$ or $T_\alpha^{-1}\bigl(H^n_{ij}(a,b)\bigr)\subseteq H^n_{ij}(-a,-b)$. Thus, for each {\fundSet} $\cH_*\subseteq\cH^n$, we can construct a {\fundSet} $\cH'_*\subseteq\cH^n$ for which $T_{\alpha}^{-1}(\bigcap\cH_*)\subseteq \bigcap\cH'_*$. It follows that each element of $\cK_n(C_\alpha^n)=T_{\alpha
    }^*\cK_n$ is contained in some element of $\cK_n$. Now condition~(i) is verified.

    \begin{figure}[h]
        \centering
        \includegraphics[width=0.7\linewidth]{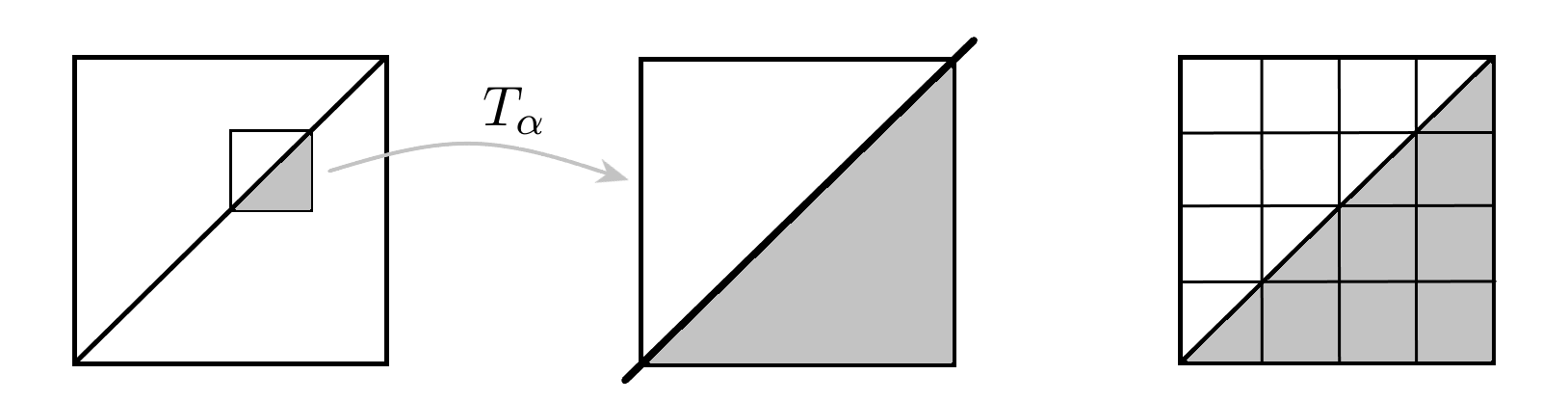}
        \caption{}
        \label{f: Knl is refinement}
    \end{figure}
	
	(ii) First, we show that each element in $\cH^n$ is a union of cells in $\cK_{n,l}$.
    Fix $H_0\in\cH^n$. It can be directly verified (see also Figure~\ref{f: Knl is refinement}) that $H_0$ is covered by $\{C^n_\alpha:\alpha\in\{0,\,\dots,\,l-1\}^n,\,x_\alpha\in H_0\}$, and that
    for each $\alpha\in\{0,\,\dots,\,l-1\}^n$ with $x_\alpha\in H_0$, $H_0\cap C^n_\alpha$ is either $C^n_\alpha$ or $T_\alpha^{-1}(H_0)$. Since it follows from Lemma~\ref{l: ort Latt: Ko Union and Intersect}~(i) that $C^n_\alpha$ and $T_\alpha^{-1}(H_0)$ are unions of cells in $\cK_{n}(C^n_{\alpha}) = T_\alpha^*\cK_n\subseteq\cK_{n,l}$, $H_0$ is a union of cells in $\cK_{n,l}$.
    Combining this with the definition of $\Ko_n$ (cf.~(\ref{e: ort Latt: definition Ko})) and Lemma~\ref{l: properties of cell decompositions}~(iv), we get that each element in $\Ko_n$ is a union of cells in $\cK_{n,l}$. 
    
    It remains to show that each element in $\cK_n$ is a union of cells in $\cK_{n,l}$. Consider $H'\in\partial\cK_n$. By Lemma~\ref{l: orthotopic Lattes: structure of dK_n}, suppose $H'=H\cap F$ for some $H\in\Ko_n$ and some facet $F$ of $\stdCube^n$. 
    Clearly $F$ is a union of cells in $\cK_{n,l}$ (which follows immediately from the definition of $\cK_{n,l}$). Thus, since $H$ is a union of cells in $\cK_{n,l}$ as shown above, by Lemma~\ref{l: properties of cell decompositions}~(iv), $H'$ is a union of cells in $\cK_{n,l}$. In conclusion, each element in $\cK_n$ is a union of cells in $\cK_{n,l}$.
\end{proof}

Now we extend the discussion on symmetric decompositions to orthotopes.

For each standard $n$-dimensional orthotope $Q$, let $\pi$ be the cubic-stretching (cf.~(\ref{e: ort Latt: cubic stretching})) of $Q$ and define
\begin{equation*}
\cK_n(Q)\=\pi^*\cK_n \quad\text{ and }\quad\cK_{n,l}(Q)\=\pi^*\cK_{n,l}
\end{equation*}
for each $l\in\N$.
Then $\cK_n(Q)$ and $\cK_{n,l}(Q)$ are cell decompositions which refine $\Rec_n(Q)$, and $\cK_{n,l}(Q)$ refines $\cK_n(Q)$. For each $n$-dimensional orthotope $R$, we define $\cK_n(R)\=L^*\cK_n(L(R))$ and $\cK_{n,k}(R)\=L^*\cK_{n,l}(L(R))$,  where $L\in \Isom(\R^n)$ maps $R$ to a standard orthotope $L(R)$. 

Note that the definitions of $\cK_n(R)$ and $\cK_{n,l}(R)$ do not depend on the choice of $L$, and this follows from the observation that for each standard orthotope $Q$, $\cK_n(Q)$ is invariant under \defn{$Q$-symmetries}, i.e., $\gamma\in \Isom(\R^n)$ for which $\gamma(Q)=Q$.
\begin{rem}\label{r: ort Latt: Q-symm -- cubic symm}
Indeed, consider $Q\=\prod_{i=1}^n[-a_i,a_i]$ and a $Q$-symmetry $\gamma\in\Isom(\R^n)$. It can be checked that $\gamma$ is also an $\stdCube^n$-symmetry, and thus, by Remark~\ref{r: ort Latt: sym(I) = DE}, there exists a permutation $\sigma\in S_n$ such that $\gamma(e_i)\in\bigl\{-e_{\sigma(i)},\,e_{\sigma(i)}\bigr\}$ for all $1 \leq  i \leq  n$. 
Moreover, $\gamma(Q)=Q$ guarantees that $a_i=a_{\sigma(i)}$ for all $1 \leq  i \leq  n$, which yields that $\pi\circ\gamma\circ\pi^{-1}\in\Sym(\stdCube^n)$.
Then since $\cK_n$ is invariant under $\stdCube^n$-symmetries (see Lemma~\ref{l: orth Latt: Ko is invariant}~(i)), we have
\begin{equation}\label{e: K(Q) is invariant}
\begin{aligned}
\gamma^*\cK_n(Q)=\gamma^*\pi^*\cK_n=\pi^*\bigl(\pi\circ\gamma\circ\pi^{-1}\bigr)^*\cK_n=\pi^*\cK_n=\cK_n(Q).
\end{aligned}
\end{equation}
Likewise, $\gamma^*\cK_{n,l}(Q)=\cK_{n,l}(Q)$ since $\cK_{n,l}$ is invariant under $\stdCube^n$-symmetries (see Remark~\ref{r: ort Latt: Kn,l is invariant}).
\end{rem}

\subsection{Cellular Markov partitions of orthotopic Latt\`es maps}\label{subsct: orthotopic Lattes are MKV}
Here we prove Theorem~\ref{t: rectangle=>Markov}.
We begin by recalling some basic results on crystallographic groups. For more details we refer the reader to \cite[Chapter~2]{SV93}. Fix a dimension $n\in\N$.
Let $\Gamma<\Isom(\R^n)$ be a crystallographic group and $P$ be a normal fundamental polyhedron of $\Gamma$. For each facet $F$ of $P$, there exists a unique $\gamma_F\in\Gamma$ for which $\gamma_F(P)\cap P=F$.\footnote{Such an element $\gamma_F$ is called an \defn{adjacency transformation} (cf.~\cite[Chapter~2]{SV93}). 
}

Let $\Gamma$ be an orthotopic crystallographic group, and $Q\=\prod_{i=1}^{n}[0,a_i]$ be a normal fundamental domain. The above discussion implies that the cover $\{\gamma(Q):\gamma\in\Gamma\}$ is given by a lattice, i.e.,
\begin{equation}\label{e: ort Latt: fund domains are lattice}
    \{\gamma(Q):\gamma\in\Gamma\}=\{\Sigma v+Q : v\in\Z^n\}, 
\text{ where } \Sigma\=\diag(a_1,\dots,a_n).
\end{equation}

In what follows, we show that each $H\in\cK_n(Q)$ is mapped by the projection $\pi\:\R^n\to\R^n/\Gamma$ homeomorphically onto its image $\pi(H)$. For this, we need the following lemma.

\begin{lemma}\label{l: rectangle Lattes: r(x) in H => r(x)=x}
	For each $H\in\cK_n$, each $x\in H$, and each $\gamma\in \Sym(\stdCube^n)$, if $\gamma(x)\in H$, then $\gamma(x)=x$.
\end{lemma}
\begin{proof}
Let $H\in\cK_n$, $x\=(x_1,\,\dots,\,x_n)\in H$, and $\gamma\in\Sym(\stdCube^n)$ satisfy $\gamma(x)\in H$.

Since, by Lemma~\ref{l: orthotopic Lattes: structure of dK_n}, each element in $\cK_n$ is contained in some element of $\Ko_n$, we may assume $H\in\Ko_n$. By the definition of $\Ko_n$ (cf.~(\ref{e: ort Latt: definition Ko})), suppose $H=\bigcap \cH_*$ for a {\fundSet} $\cH_*\subseteq\cH^n$ (cf.~(\ref{e: ort Latt: fund subset})).
By Lemma~\ref{l: orthotopic Lattes: property of K^o_n}~(i), we may assume $H\subseteq [0,1]^n$.
	
	By Remark~\ref{r: ort Latt: sym(I) = DE}, suppose $\gamma(e_i)=\delta_i e_{\tau(i)}$, $\delta_i\in\{-1,\,1\}$, for each $1 \leq  i \leq  n$, where $\tau\in S_n$ is a permutation. Set $\tau=\sigma_1\cdots \sigma_k$, where $\sigma_i$, $1 \leq  i \leq  k$, are disjoint cycles. 
    
	Since $x=\sum_{i=1}^nx_ie_i$ and $\gamma(x)=\sum_{i=1}^n\delta_i x_ie_{\tau(i)}$ are in $H\subseteq[0,1]^n$, we have $\gamma(x)=\sum_{i=1}^nx_ie_{\tau(i)}$. Thus, to show $\gamma(x)=x$, it suffices to show, for each $1\leq j\leq k$, that $\sum_{i=1}^n x_ie_i=\sum_{i=1}^nx_ie_{\sigma_j(i)}$.

    Let $\sigma\=(i_1 i_2\dots i_l)$ be an arbitrary cycle among $\sigma_i$, $1\leq i\leq k$. We enumerate $\bigl\{i'_1,\,\dots,\,i'_{n-l}\bigr\}\=\{1,\,\dots,\,n\}\smallsetminus\{i_1,\,\dots,\,i_l\}$. Then $x=\sum_{i=1}^n x_ie_i=\sum_{s=1}^{l}x_{i_s}e_{i_s}+\sum_{t=1}^{n-l}x_{i'_t}e_{i'_t}$. Since $\sigma_i$, $1 \leq  i \leq  k$, are disjoint cycles, we have $\bigl\{\tau(i'_1),\,\dots,\,\tau \bigl( i'_{n-l} \bigr)\bigr\}=\bigl\{i'_1,\,\dots,\,i'_{n-l}\bigr\}$ and
    \begin{equation}\label{e: ort Latt: cycle of x_i}
        \gamma(x)=\sum_{i=1}^n x_i e_{\tau(i)}=\sum_{s=1}^{l}x_{i_s}e_{\sigma(i_s)}+\sum_{t=1}^{n-l}x_{i'_t}e_{\tau(i'_t)}=\sum_{s=1}^{l}x_{i_s}e_{i_{s+1}}+\sum_{t=1}^{n-l}x_{i'_t}e_{\tau(i'_t)},
    \end{equation}
    where we follow the convention that $i_{l+1}\=i_1$.
    
    We show $x_{i_1}=x_{i_2}=\cdots=x_{i_l}$.
	We argue by contradiction and assume (without loss of generality) that $x_{i_1}\neq x_{i_2}$. First, assume $x_{i_1}>x_{i_2}$. 
    Under such an assumption, suppose $x_{i_1}>x_{i_j}$ for some $2 \leq  j \leq  l$. By~(\ref{e: ort Latt: cycle of x_i}) and (\ref{e: ort Latt: Hi and Hij}), $\gamma(x)\notin H^n_{i_2 i_{j+1}}(-1,-1)$. Then since $\gamma(x)\in H=\bigcap\cH_*$, we have $H^n_{i_2 i_{j+1}}(-1,-1)\notin\cH_*$, and thus, by the definition of {\fundSet}s (cf.~(\ref{e: ort Latt: fund subset})), $H^n_{i_2 i_{j+1}}(1,1)\in\cH_*$, 
    which implies
	$x_{i_1}>x_{i_2} \geq  x_{i_{j+1}}$. This, inductively, yields $x_{i_1}>x_{i_{l+1}}=x_{i_1}$, which is a contradiction.
    Similarly, $x_{i_1}<x_{i_2}$ leads to a contradiction. 

    Hence, $x_{i_1}=x_{i_2}=\cdots=x_{i_l}$ and $\sum_{i=1}^n x_ie_i=\sum_{i=1}^nx_ie_{\sigma(i)}$.
	Since $\sigma$ is an arbitrary cycle among $\sigma_1,\,\dots,\,\sigma_k$, we have $\gamma(x)=x$.
\end{proof}

\begin{cor}\label{c: rectangle Lattes: r(x) in H => r(x)=x}
	Let $R$ be an orthotope of dimension $n\in\N$, and $\gamma$ be a $R$-symmetry. If $H\in\cK_n(R)$ and $x\in H$ satisfy $\gamma(x)\in H$, then $\gamma(x)=x$.
\end{cor}
\begin{proof}
	Let $L\in \Isom(\R^n)$ maps $R$ to a standard orthotope $Q\=L(R)$. Let $\pi$ be the cubic stretching of $Q$. Then $H_1\=\pi(L(H))\in\cK_n$ and $x_1\=\pi(L(x))\in H_1$ satisfy 
	$\gamma_1(x_1)=\bigl(\pi\circ\bigl(L\circ\gamma\circ L^{-1}\bigr)\circ\pi^{-1}\bigr)(\pi(L(x)))=\pi(L(\gamma(x)))\in H_1$. Clearly $L\circ\gamma\circ L^{-1}$ is a $Q$-symmetry, and $\gamma_1\=\pi\circ L\circ\gamma\circ L^{-1}\circ\pi^{-1}\in\Sym(\stdCube^n)$ (cf.~Remark~\ref{r: ort Latt: Q-symm -- cubic symm}).
	Then by Lemma~\ref{l: rectangle Lattes: r(x) in H => r(x)=x}, $\gamma_1(x_1)=x_1$, and thus $\gamma(x)=x$.
\end{proof}

\begin{lemma}\label{l: rectangle Lattes: H cap Gx = x}
	Let $\Gamma$ be an orthotopic crystallographic group and $Q$ be an orthotope that is a normal fundamental domain of $\Gamma$. Then for each $H\in\cK_n(Q)$ and each point $x\in H$, we have $H\cap\Gamma x=\{x\}$.
\end{lemma}
\begin{proof}
	Fix $H\in\cK_n(Q)$ and $x\in H$. Let $\gamma\in\Gamma$ satisfy $\gamma(x)\in H\cap\Gamma x$.
	By the definition of normal fundamental domains, $\gamma(Q)\cap Q$ is a common face $c\in\Rec_n(Q)\cap\Rec_n(\gamma(Q))$. Let $c^*\in\Rec_n(Q)$ be such that $x\in c^*$ and $\gamma(c^*)=c$. Since $\gamma$ is an isometry, we can extend $\gamma|_{c^*}$ to a $Q$-symmetry $\tgamma$. Then $\tgamma(x)=\gamma(x)\in H$.
	By Corollary~\ref{c: rectangle Lattes: r(x) in H => r(x)=x}, we have $\gamma(x)=\tgamma(x)=x$.
\end{proof}

Let $(\Gamma,h,A)$ be an orthotopic Latt\`es triple on an {\orclcnR} $n$-manifold $\mfd^n$, and $Q\=\prod_{i=1}^n[0,a_i]$ be a fundamental domain of $\Gamma$. Suppose $A=\lambda I$ for some $\lambda\in\N$. Now we show that $\cK_n(Q)$ and $\cK_{n,\lambda}(Q)$ induce, via $h$, cell decompositions of $\mfd^n$. 

\begin{prop}\label{p: ort Latt: induced cell decomp}
	Let $(\Gamma,h,A)$ be an orthotopic Latt\`es triple on an {\orclcnR} $n$-manifold $\mfd^n$, where $\Gamma$ has a fundamental domain $Q\=\prod_{i=1}^n[0,a_i]$, and $A=\lambda I$ for some $\lambda\in\N$. 
    Then the restriction of $h$ to each cell $c$ in $\cK_n(Q)\cup\cK_{n,\lambda}(Q)$ is homeomorphic. Moreover,
	\begin{equation*}
	\cD_0\=\{h(c):c\in\cK_n(Q)\}\quad\text{ and }\quad\cD_1\=\{h(c):c\in\cK_{n,\lambda}(Q)\}
\end{equation*}
	are cell decompositions of $\mfd^n$, and $\cD_1$ is a refinement of $\cD_0$.
\end{prop}
\begin{proof}
	By Lemma~\ref{l: rectangle Lattes: H cap Gx = x}, each $c\in\cK_n(Q)$ contains at most one element in each orbit of $\Gamma$. Then $h|_c$ is injective. Since $c$ is compact, $h|_c$ is also closed (since a continuous map between compact Hausdorff spaces is closed; cf.~\cite[Section~26]{Mu00}), and thus a homeomorphism. Hence, $\cD_0$ is indeed a collection of cells.

    Since $\Gamma$ is cocompact, we have $h(\R^n)=\mfd^n$ (which is a direct consequence of~\cite[Theorems~1.5 and~1.7]{Ka22}). 
	Since $Q$ is a fundamental domain of $\Gamma$, we have $h(Q)=h(\R^n)=\mfd^n$. Then $\bigcup\cD_0=\mfd^n$, which verifies condition~(i) in Definition~\ref{d: cell decomposition}.
    Conditions~(iii) and (iv) in Definition~\ref{d: cell decomposition} follow directly from the fact that $\cK_n(Q)$ is a cell decomposition of finite cardinality. 
    
    It remains to verify condition~(ii).
	Let $\sigma,\,\tau\in\cD_0$ satisfy $\brCellint{\sigma}\cap\brCellint{\tau}\neq\emptyset$ and suppose $\sigma=h(\sigma_0)$, $\tau=h(\tau_0)$ for some $\sigma_0,\,\tau_0\in\cK_n(Q)$. 
	Fix $p\in\brCellint{\sigma}\cap\brCellint{\tau}$, and let $x\in\stdCellint{\sigma_0}$, $y\in\stdCellint{\tau_0}$ be such that $h(x)=h(y)=p$. Since $\cK_n(Q)$ is a refinement of $\Rec_n(Q)$, by Lemma~\ref{l: cell decomp: refinement inte() contained in inte()} there exist $X,\,Y\in \Rec_n(Q)$ for which $\stdCellint{\sigma_0}\subseteq\stdCellint{X}$ and $\stdCellint{\tau_0}\subseteq\stdCellint{Y}$.

    Since $h(x)=h(y)=p$, by~Definition~\ref{d: Lattes triple} there exists $\gamma\in\Gamma$ such that $\gamma(x)=y$. The choice of $X,\,Y$ yields $\gamma(x)=y\in\brCellint{\gamma(X)}\cap\brCellint{Y}$.
    By the definition of normal fundamental domains, $\gamma(Q)\cap Q$ is a common face $Z\in\Rec_n(Q)\cap\Rec_n(\gamma(Q))$, which guarantees that $\Rec_n(Q)\cup\Rec_n(\gamma(Q))$ is a cell decomposition of $Q\cup\gamma(Q)$.
	Then $\gamma(X)=Y$ since $\brCellint{\gamma(X)}\cap\brCellint{Y}\neq\emptyset$. 
	Extend $\gamma|_X$ to a $Q$-symmetry $\tgamma$. The invariance of $\cK_n(Q)$ under $Q$-symmetries (cf.~(\ref{e: K(Q) is invariant})) yields $\gamma(\sigma_0)=\tgamma(\sigma_0)\in\cK_n(Q)$. Thus, since $y=\gamma(x)\in\brCellint{\gamma(\sigma_0)}\cap\brCellint{\tau_0}\neq\emptyset$, we have $\gamma(\sigma_0)=\tau_0$. Then since $h\circ\gamma=h$ (cf.~Definition~\ref{d: Lattes triple}), $\sigma=h(\sigma_0)=h(\gamma(\sigma_0))=h(\tau_0)=\tau$.
	 
	Therefore, $\cD_0$ is a cell decomposition of $\mfd^n$. By the same arguments, $\cD_1$ is also a cell decomposition. Since $\cK_{n,\lambda}(Q)$ refines $\cK_n(Q)$, $\cD_1$ refines $\cD_0$.
\end{proof}

Now we obtain a cellular Markov partition (cf.~Definition~\ref{d: cellular Markov}) of the Latt\`es map induced by an orthotopic Latt\`es triple.

\begin{proof}[Proof of Theorem~\ref{t: rectangle=>Markov}]
	Consider an orthotopic Latt\`es map $f\:\mfd^n\to\mfd^n$ with respect to an orthotopic Latt\`es triple $(\Gamma,h,A)$, where $\Gamma$ has a fundamental domain $Q\=\prod_{i=1}^n[0,a_i]$, and $A=\lambda I$ for some $\lambda\in\N$.
    Set $\cD_0\=\{h(c):c\in\cK_n(Q)\}$ and $\cD_1\=\{h(c):c\in\cK_{n,\lambda}(Q)\}$.

	By Proposition~\ref{p: ort Latt: induced cell decomp}, $D_0$ and $D_1$ are cell decompositions of $\mfd^n$ for which $D_1$ refines $D_0$. It remains to show that $f$ is $(\cD_1,\cD_0)$-cellular.
	
	Let $c\in\cD_1$ be arbitrary. Suppose $c=h(\sigma)$ for some $\sigma\in\cK_{n,\lambda}$. By~(\ref{e: ort Latt: fund domains are lattice}) and the construction of $\cK_{n,\lambda}(Q)$ (cf.~(\ref{e: ort Latt: K_n,l})), it is easy to see that $\tau\=A(\sigma)\in\cK_n(\gamma(Q))$ for some $\gamma\in\Gamma$.  Then since $h=h\circ\gamma^{-1}$ (cf.~Definition~\ref{d: Lattes triple}) and by Proposition~\ref{p: ort Latt: induced cell decomp}, $h|_\tau=\bigl(h|_{\gamma^{-1}(\tau)}\bigr)\circ\bigl(\gamma^{-1}|_\tau\bigr)$ is a homeomorphism for which $h(\tau)= h\bigl(\gamma^{-1}(\tau)\bigr)\in\cD_0$.
    Hence, since $h|_\sigma:\sigma\to c$ is homeomorphic (see Proposition~\ref{p: ort Latt: induced cell decomp}) and $f\circ h=h\circ A$ (cf.~Definition~\ref{d: Lattes triple}), $f|_c=h|_{\tau}\circ A|_{\sigma}\circ(h|_c)^{-1}$ is a homeomorphism between $c\in\cD_1$ and $f(c)=h(\tau)\in\cD_0$. 
\end{proof}


\end{document}